\pdfoutput=1
\documentclass{amsart}
\usepackage{graphicx,color}
\usepackage{amssymb}

\newcommand{\rii}{\operatorname{RI\!I}}

\theoremstyle{theorem}
\newtheorem{fact}{Fact}
\newtheorem{theorem}{Theorem}
\newtheorem{claim}{Claim}

\newtheorem{lemma}{Lemma}
\newtheorem{proposition}{Proposition}
\newtheorem{corollary}{Corollary}
\newtheorem{ostlund}{\"{O}stlund Conjecture}

\newtheorem{haggeyazinski}{Hagge-Yazinski Theorem}

\theoremstyle{remark}
\newtheorem{remark}{Remark}

\theoremstyle{definition}
\newtheorem{definition}{Definition}
\newtheorem{notation}{Notation}

\begin{document}

\title{RI\!I number of knot projections}
\author{Noboru Ito}
\address{The University of Tokyo, 3-8-1, Komaba, Meguro-ku, Tokyo, 153-8914, Japan}
\address{Current address: National Institute of Technology, Ibaraki College, 866, Nakane,  Hitachinaka, Ibaraki, 312-8508, Japan}
\email{nito@ibaraki-ct.ac.jp}
\author{Yusuke Takimura}
\address{Gakushuin Boys' Junior High School, 1-5-1,  Mejiro, Toshima-ku, Tokyo, 171-0031, Japan}
\email{Yusuke.Takimura@gakushuin.ac.jp}
\date{Accepted August 1, 2019}
\maketitle
\begin{abstract}
Every knot projection is simplified to the trivial spherical curve not increasing double points by using deformations of types 1, 2, and 3 which are analogies of Reidemeister moves of types 1, 2, and 3 on knot diagrams.   We introduce $\rii$ number of a knot projection that is the minimum number of deformations of negative type~2 among such sequences.  By definition, it is invariant under deformations of types 1 and 3.  
This is motivated by \"{O}stlund conjecture:      
Deformations of type~1 and 3 are sufficient to describe a homotopy from any generic immersion of a circle in a two dimensional plane to an embedding of the circle (2001), which implies $\rii$ number always would be zero.     
However, Hagge and Yazinski disproved the conjecture by showing the first counterexample with $16$ double points, which implies that $\rii$ number is nontrivial.  This paper shows that $\rii$ number can be any nonnegative number.        
\end{abstract}
\section{Introduction}
A {\it{knot projection}} is the image of a generic immersion of a circle into the $2$-sphere.  Two knot projections are identified by ambient isotopy.  Any self-intersection of a knot projection is a transverse double point and is simply called a {\it{double point}}.  The {\it{trivial spherical curve}} is a knot projection with no double points.  
Deformations of types 1, 2, and 3 are local replacements defined in Fig.~\ref{1301}.  
\begin{figure}[h!]
\includegraphics[width=12cm]{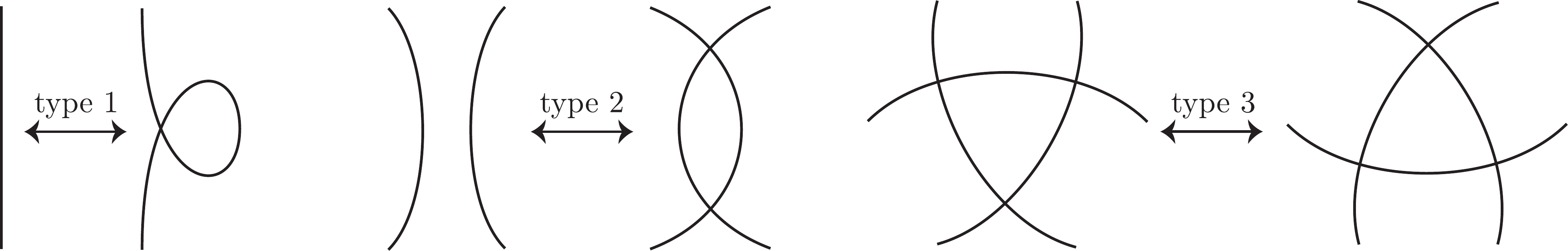}
\caption{Deformations of types 1, 2, and 3}\label{1301}
\end{figure}
This deformations are analogies of Reidemeister moves of knot diagrams.   
Every knot projection is related to the trivial spherical curve by a finite sequence of deformations of types 1, 2, and 3.  

In 2001 \cite{O}, \"{O}stlund formulated a conjecture in the following: 
\begin{ostlund}\label{ostlund_conj}
Deformations of types 1 and 3 are sufficient to obtain a homotopy from any generic immersion $S^{1}$ $\to$ $\mathbb{R}^2$ to an embedding.  
\end{ostlund}
In 2014, Hagge and Yazinski \cite{HY} disproved this conjecture as follows:  
\begin{haggeyazinski}\label{hythm}
For $P_{HY}$ that appears as Fig.~\ref{1302}, there is no finite deformations of types 1 and 3 from $P_{HY}$ to the trivial spherical curve up to ambient isotopies.     
\begin{figure}[h!]
\includegraphics[width=3cm]{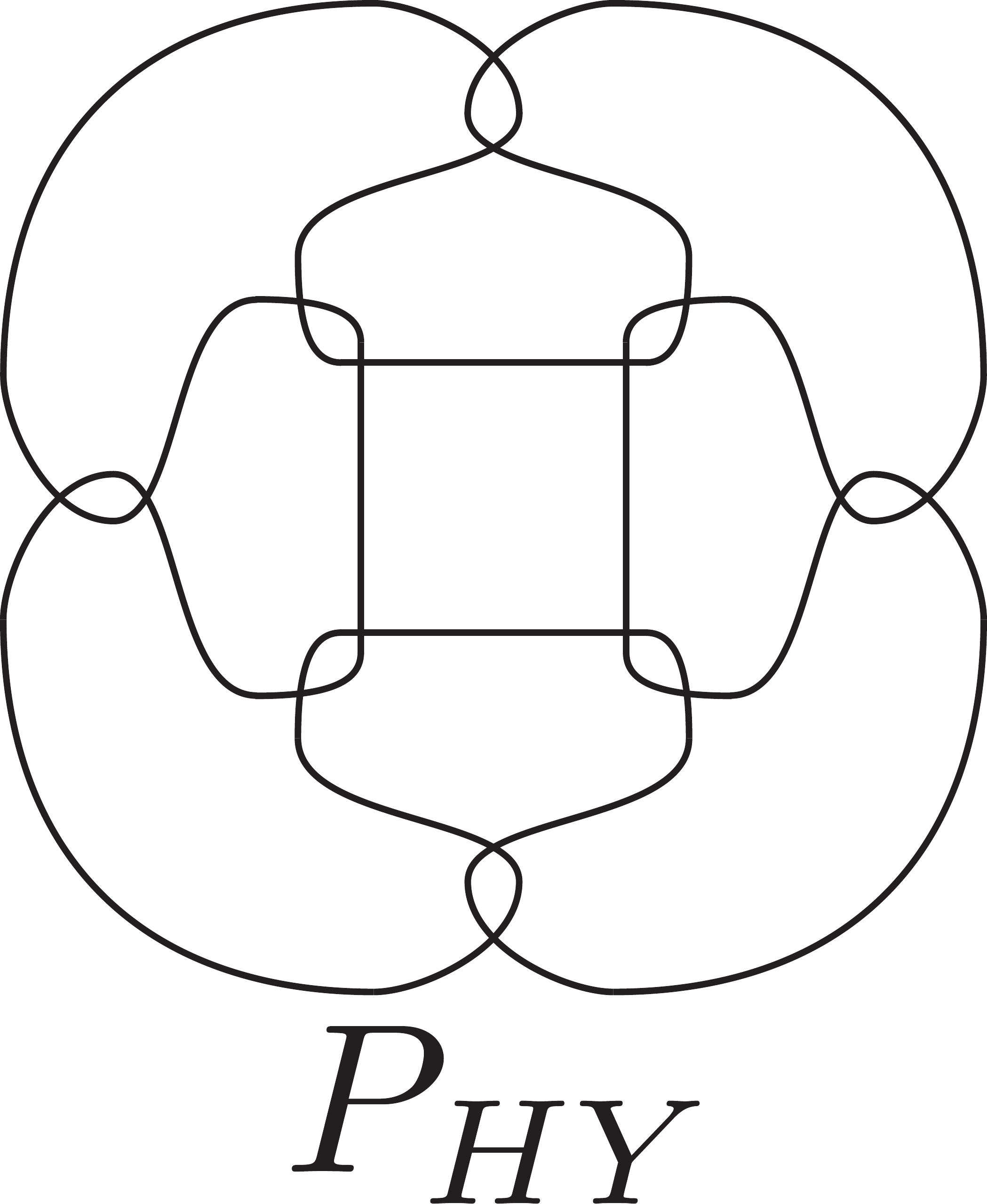}
\caption{Hagge-Yazinski's example $P_{HY}$ having  16 double points}\label{1302}
\end{figure}
\end{haggeyazinski}

In 2016, in \cite{IT15}, the authors obtain a generalization of the Hagge-Yazinski Theorem, that is, we show that there exists an infinite family of knot projections,  which are counterexamples of the \"{O}stlund Conjecture for any $n$ double points ($n \ge 15$, for 15 double points, see Fig.~\ref{1302a}),    
including $P(l, m, n)$ \cite[Remark~1, Page 28, Fig.~13]{IT15}.   Throughout of this paper, let $P(m, n)$ $=$ $P(1, m, n)$ (Fig.~\ref{1303}).  

\begin{figure}[h!]
\includegraphics[width=3cm]{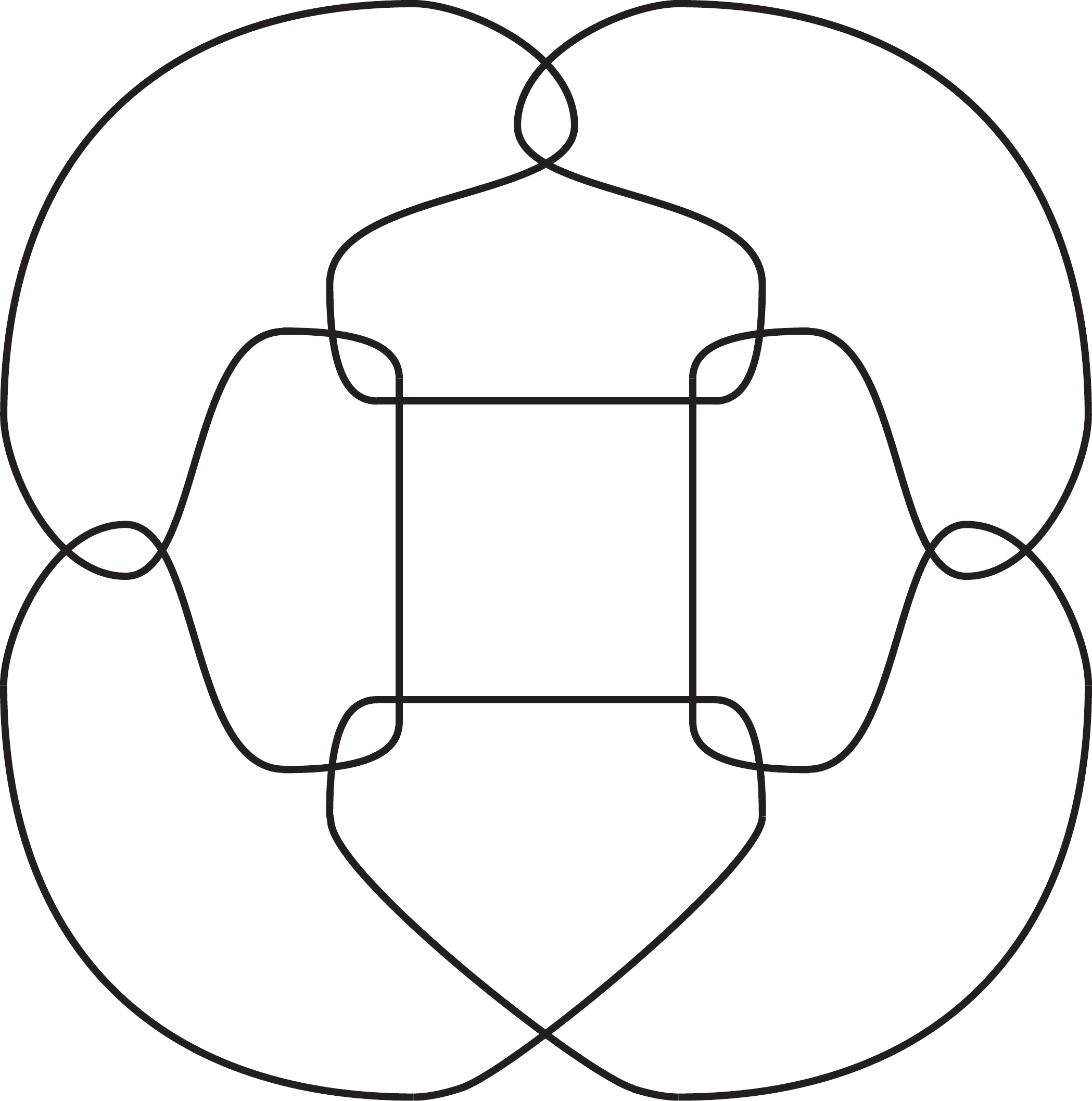}
\caption{Our example having 15 double points}\label{1302a}
\end{figure}

We call a deformation of type 2 decreasing double points a deformation of \emph{negative type} 2.  Then the next fact is well known to the experts.  
\begin{fact}[\cite{L}]\label{L_fact}
Every knot projection is related to the trivial spherical curve by a finite sequence of deformations of types 1, negative~2, and 3.
\end{fact}
In this paper, for a given knot projection $P$, we introduce \emph{$\rii$ number} that 
is the minimum number of deformations of negative type~2 among sequences, each of which consists of deformations of type~1,  negative type~2, and type~3.  
This number is denoted by $\rii(P)$.    
By definition, \"{O}stlund conjecture implies $\rii$ number always would be zero.   However, Hagge-Yazinski Theorem implies that $\rii$ number is nontrivial.   
In Section~\ref{application}, we study $2$ classes of knot projections called pretzel knot projections, and two-bridge knot projections, and show that every such knot projection $P$  satisfies $\rii(P)=0$ (Propositions~\ref{proposition1} and \ref{proposition2}).     
Theorem~\ref{main1} implies that for any integer $m$ $(\ge 1)$, there exists $P$ such that $\rii(P)=m$.  

Two knot projections are \emph{(1, 3) homotopic} if they are related by finite deformations of types 1 and 3 and ambient isotopies.  The relation becomes an equivalence relation and is called {\it{(1, 3) homotopy}}.  
\begin{proposition}\label{prop}
Let $P$ be a knot projection.  The number $\rii(P)$ is an invariant under (1, 3) homotopy. 
\end{proposition}
\begin{theorem}\label{main1} 
For positive integers $m \ge 1$, and $n \ge 4$, let $P(m, n)$ be a knot projection, as in Fig.~\ref{1303}.   
Then, $\rii(P(m, n)) = m$ for any $(m, n)$. 
\end{theorem}
\begin{figure}[h!]
\includegraphics[width=12cm]{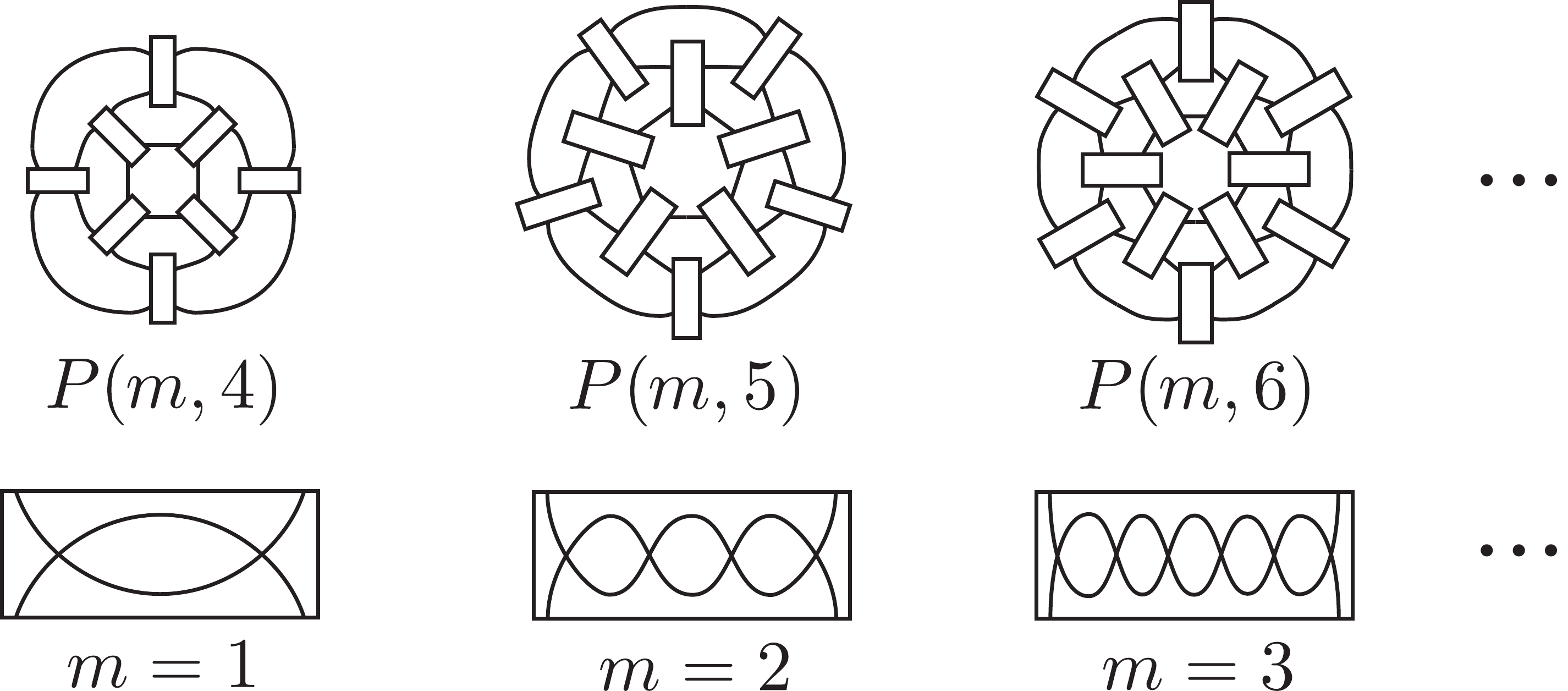}
\caption{$P(m, n)$.  Note that $P(1, 4)$ $=$ $P_{HY}$.}\label{1303}
\end{figure}
\begin{corollary}
There exists infinitely many (1, 3) homotopy classes of knot projections.  
\end{corollary}
For a proof of Theorem~\ref{main1}, see Section~\ref{sec_proof}.  

\section{Preliminaries}
\begin{definition}[$\rii$ number $\rii(P)$]
Let $P$ be a knot projection.  The \emph{$\rii$ number}  
is the minimum number of deformations of negative type~2 among sequences, each of which consists of deformations of type~1, negative type~2, and type~3.  
The number is denoted by $\rii(P)$.  
\end{definition}
\begin{definition}[($3, 3$)-tangle]
An unoriented ($3, 3$)-tangle is the image of a generic immersion of $3$ arcs into $[0, 1] \times [0, 1]$ such that:
\begin{itemize}
\item The boundary points of the arcs map bijectively to $6$ points \[ \left \{\frac{1}{4}, \frac{2}{4}, \frac{3}{4} \right \} \times \{ 1 \}, \left \{ \frac{1}{4}, \frac{2}{4}, \frac{3}{4} \right \} \times \{ 0 \}.\]  
\item Near the endpoints, the arcs are perpendicular to the boundary $[0, 1]$.  
\end{itemize}
In a (3, 3)-tangle, we call an image of the map of a single arc a {\it{strand}}.  
\end{definition} 
\begin{definition}
Let $P$ be a knot projection and $F$ the closure of a connected component in $S^2 \setminus P$.  Let $n$ be a positive integer.  Suppose that the double points of $P$ that lie on $\partial F $ are removed, the reminder consists of $n$ connected components, each of which is homeomorphic to an open interval.  Then, $\partial F$ is called an \emph{$n$-gon}.  When we do not specify $n$, an $n$-gon is called a \emph{polygon}.  
\end{definition} 
\begin{notation}[\cite{FHIKM}]\label{notation1}
Let $P$ and $P'$ be two knot projections that are equivalent under deformations of type~1 and type~3, i.e., there exists a finite sequence of knot projections $P$ $=$ $P_0$, $P_1$, \ldots, $P_m$ $=$ $P'$, where $P_i$ is obtained from $P_{i-1}$ by a deformation of type~1 or type~3.  Then, $Op_i$ denotes the deformation from $P_{i-1}$ to $P_i$, and the setting are expressed by using the notation: 
\[
P = P_0 \stackrel{Op_1}{\to} P_1 \stackrel{Op_2}{\to} \dots \stackrel{Op_m}{\to} P_m = P'.  
\]
\end{notation}

\section{Proofs of Proposition~\ref{prop} and Theorem~\ref{main1}}\label{sec_proof}
Before starting the proof, we explain our plan of the proof.  
\begin{itemize}
\item We show that $\rii(P)$ is invariant under (1, 3) homotopy.  \label{pf1}
\item We show that  $P(m, n) \ge m$.   Most of the part is obtained from a similar proof of Hagge-Yazinski Theorem \cite{HY} or \cite{IT15}.   Therefore, a reader who is familiar to \cite{HY} can skip this part except for Section~\ref{36}.   Since it is elementary to prove it, we prove it here.  We use the terminology used in \cite{HY}.  \label{pf2}
\item We show that $P(m, n) \le m$.  We introduce new techniques and show this part.  \label{pf3}
\end{itemize}

\noindent $\bullet$ {\bf{Proof of Proposition~\ref{prop}.}}  
Let $P$ and $P'$ be a pair of (1, 3) homotopic knot projections.    
Let $m=\rii(P')$, hence, there exists a sequence of deformations consisting of deformations of type~1, negative type~2, or type~3, and it contains exactly $m$ deformations of negative type~2. By combining the deformations from $P$ and $P'$, and $P'$ to the trivial spherical curve, we obtain a sequence from $P$ to the trivial spherical curve which contains exactly $m$ deformities of negative type~2.  This shows that $\rii(P) \le \rii(P')$.  Since the argument is symmetric, we see that $\rii(P) \ge \rii(P')$ holds, too.  Hence, $\rii(P) = \rii(P')$.   

\subsection{Proof of $P(m, n) \ge 1$.}    
For any $P(m, n)$, there exist $2n$ boxes  such that the intersections of $P(m, n)$ and $2n$ boxes are $(3, 3)$-tangles, and there are no double points outside the boxes (Figs.~\ref{1303} and \ref{1305}).    Each box together with the portion of $P(m, n)$  contained in the box is equivalent to $[0, 1] \times [0, 1]$ corresponding to a (3, 3)-tangle.  For example, for the case $m=1$ and $n=4$, $8$ boxes are shown in Fig.~\ref{1305}.  
Each $P(m, n)$ satisfies the following two  conditions:
\begin{figure}[h!]
\includegraphics[width=8cm]{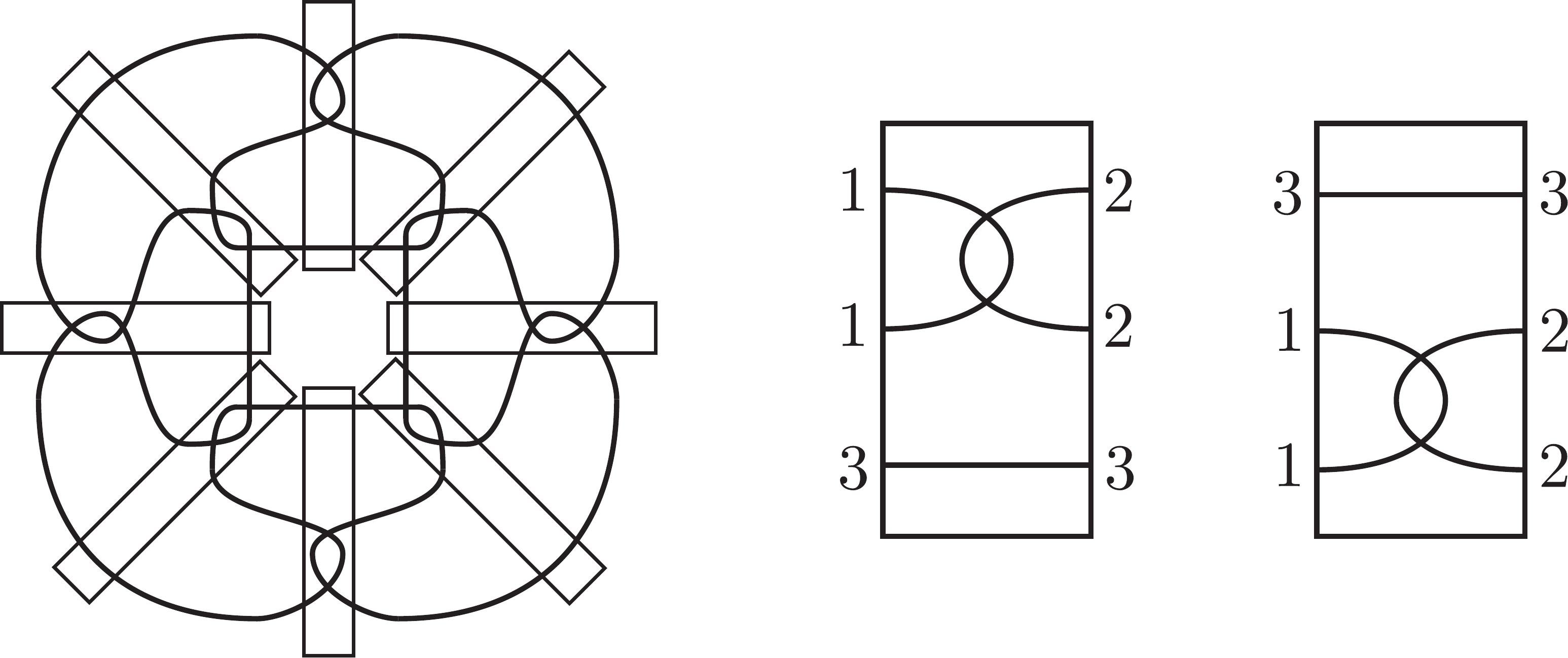}
\caption{$P(1, 4)$ with boxes (left), strands $1$, $2$, and $3$ (right)}\label{1305}
\end{figure}
\begin{enumerate}
\item No double points are placed outside the boxes. Three arcs connect two adjacent boxes concentrically.  
There exist exactly two polygons, 
each of which has at least $n$ sides partially outside boxes.  \label{r3}
\item If we fix our gazing direction from infinity, which is selected as shown Figs.~\ref{1303} and \ref{1305}, we define the {\it{left-side}} and {\it{right-side}} of each box.  In each box, strand~1 ($2$,~resp.) is a strand that begins and ends on the left-side (right-side,~resp.).  In each box, strand~3 has one endpoint on the left-side and another endpoint the right-side.  There exist $2n$ pairs of strand~1 and strand $2$; for each pair, strand~1 and strand~2 intersect at exactly $2m$ double points.  
Strand~1 of a pair cannot intersect strand~2 belonging to a different pair.   
\label{r1}

\end{enumerate}

Let $m$ and $n$ be positive integers ($m \ge 1, n \ge 4$).  
In general, for a knot projection $Q$, we say that $Q$ satisfies ($2m, 2n$) \emph{box property} if there exist $2n$ boxes $B_i$ ($1 \le i \le 2n$) in $S^2$, as shown in Fig.~\ref{1305b}, satisfying the following (A) and (B).  
\begin{enumerate}
\item[(A)] There exist no double points of $Q$ in $S^2 \setminus (B_1 \cup B_2 \cup \dots \cup B_{2n})$, and $Q \cap (S^2 \setminus (B_1 \cup B_2 \cup \dots \cup B_{2n}))$ consists of $3 \times 2n$ simple arcs as in Fig.~\ref{1305b}.  \label{R3}
\begin{figure}[h!]
\includegraphics[width=4cm]{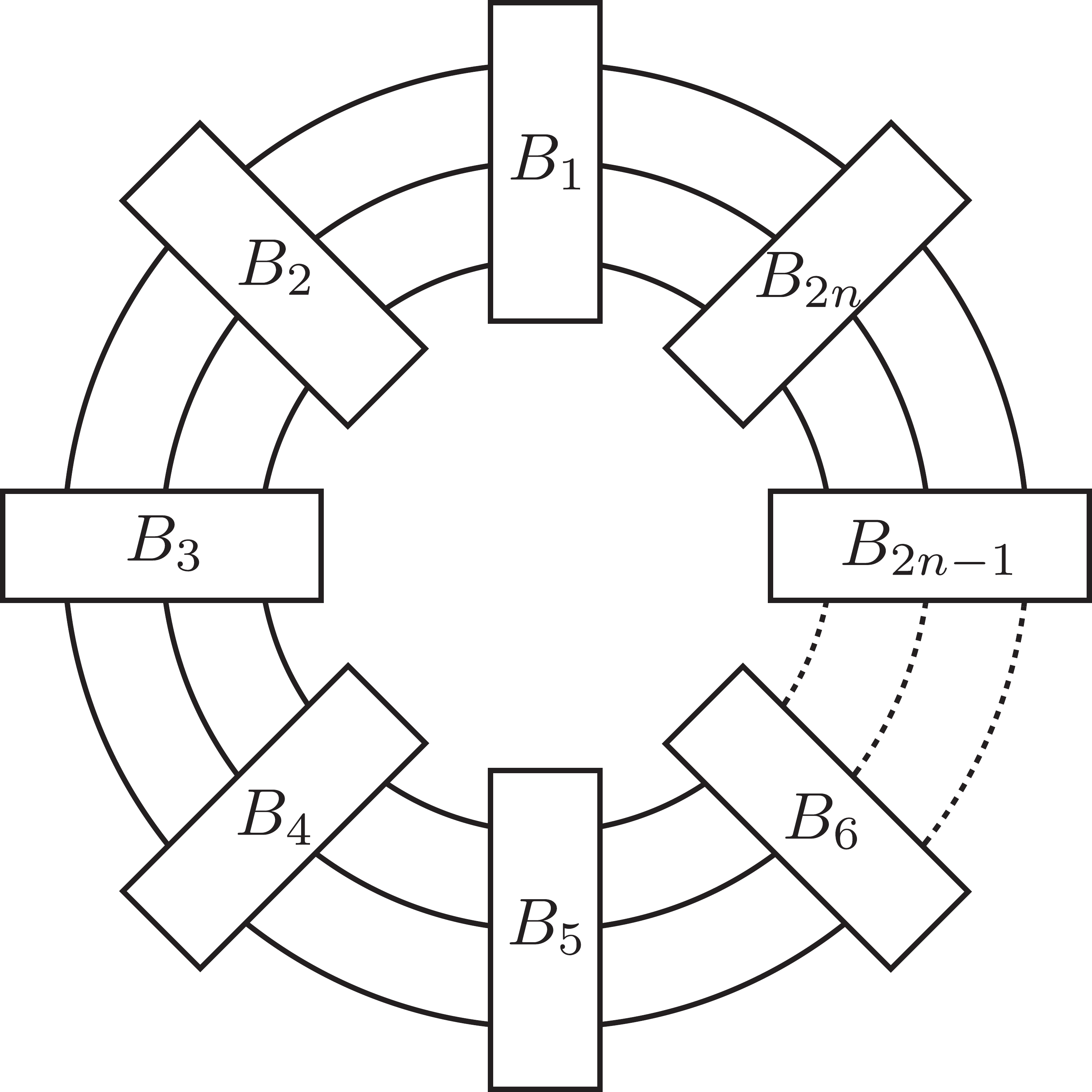}
\caption{Boxes $B_1$, $B_2$, \dots , $B_{2n}$}\label{1305b}
\end{figure}
\item[(B)] For each $i$ ($1 \le i \le 2n$), $Q \cap B_i$ is a (3, 3)-tangle that is the union of three immersed arcs.  Then it satisfies the following condition.  \label{R1}
\begin{itemize}
\begin{figure}[h!]
\includegraphics[width=10cm]{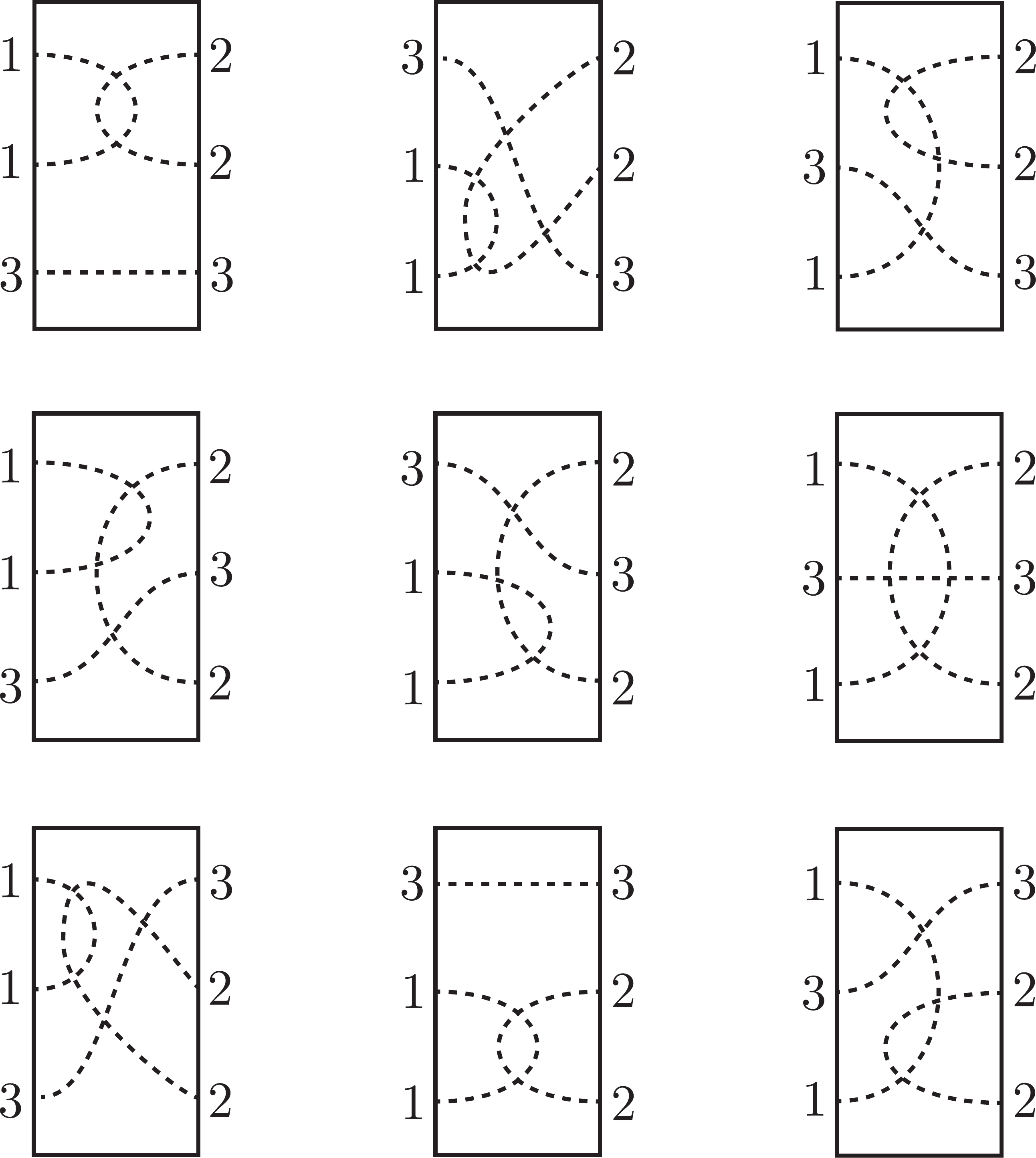}
\caption{The types of $B_i$}\label{16}
\end{figure}
\item The endpoints of the three arcs are located on $\partial B_i$ as in one of the figures of  Fig.~\ref{16}, where we name the immersed arcs 1, 2, and 3 as in Fig.~~\ref{16}.  Further, there exist at least $2m$ double points formed by a subarc of 1, and a subarc of 2.  
\end{itemize}  
\end{enumerate}

\noindent $\bullet$ {\it{Proof of $\rii(P(m, n)) \ge 1$.}}
We consider an inductive proof with respect to the number of deformations, type~1 or 3, applied to $P(m, n)$.  This induction proves  Claim~\ref{claim_A}, which implies that $P(m, n)$ cannot be (1, 3) homotopic to the trivial spherical curve; that is, $\rii(P(m, n)) \ge 1$.  Recall Notation~\ref{notation1}.  
\begin{claim}\label{claim_A}
Let $Q$ be a knot projection satisfying a $(2m, 2n)$ box property.  Then, for each sequence of knot projections   
\[
Q = Q_0 \stackrel{Op_1}{\to} Q_1 \stackrel{Op_2}{\to} \dots \stackrel{Op_r}{\to} Q_r 
\]
such that each $Op_i$ $(1 \le i \le r)$ is of type 1 or 3, we have:

by retaking the boxes if necessary, each $Q_i$ $(1 \le i \le r)$ satisfies $(2m, 2n)$ box property.  
\end{claim}

First, for $Q_0$, it is clear that Claim~\ref{claim_A} holds.  Second, we suppose that Claim~\ref{claim_A} holds for $Q_{r-1}$ and prove it $Q_r$.     

\subsection{On a deformation of type~1 or type~3 occurring within a box}
Suppose that the $Op_{r}$ is a single deformation (type~1 or type~3) occurring entirely within a box.  This deformation fixes the endpoints of the strands, and thus Claim~\ref{claim_A} holds.   

\subsection{On a deformation of type~1 not occurring within a box and increasing double points}
Suppose that $Op_{r}$ is a single deformation of type~1 increasing double points.    
If the new $1$-gon produced by $Op_{r}$ is outside the boxes, by retaking the box by a sphere isotopy, as shown in Fig.~\ref{1307}, $Op_{r}$ is entirely within a box.  If the new $1$-gon produced by $Op_{r}$ is not completely outside the boxes, a similar modification works to retake the box.       
\begin{figure}[h!]
\includegraphics[width=10cm]{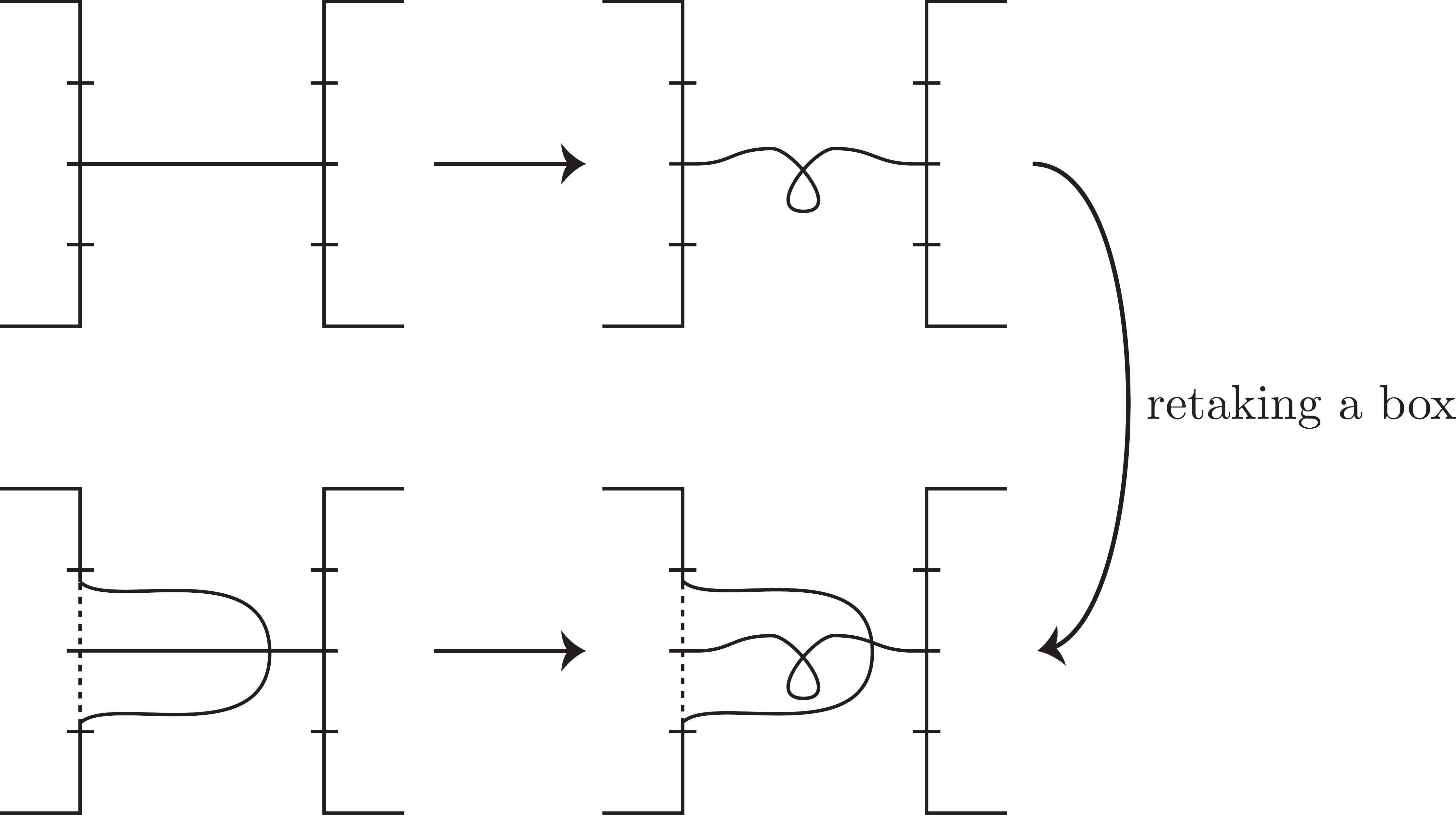}
\caption{Retaking a box}\label{1307}
\end{figure}
\subsection{On a deformation of type~1 not occurring within a box and decreasing double points}
Suppose that the $Op_{r}$ does not occur within a box.  Since $Op_r$ is a deformation of type~1 decreasing double points, there exists a $1$-gon to be removed in $Q_{r-1}$.  The two possibilities of appearing of a $1$-gon are considered as follows.
\begin{itemize}
\item Suppose that the $1$-gon contains a region having one side, as shown in Fig.~\ref{1308a}.  By the induction assumption of $P_{r-1}$, the region has at least four sides,  which is a contradiction. 
\begin{figure}
\includegraphics[width=5cm]{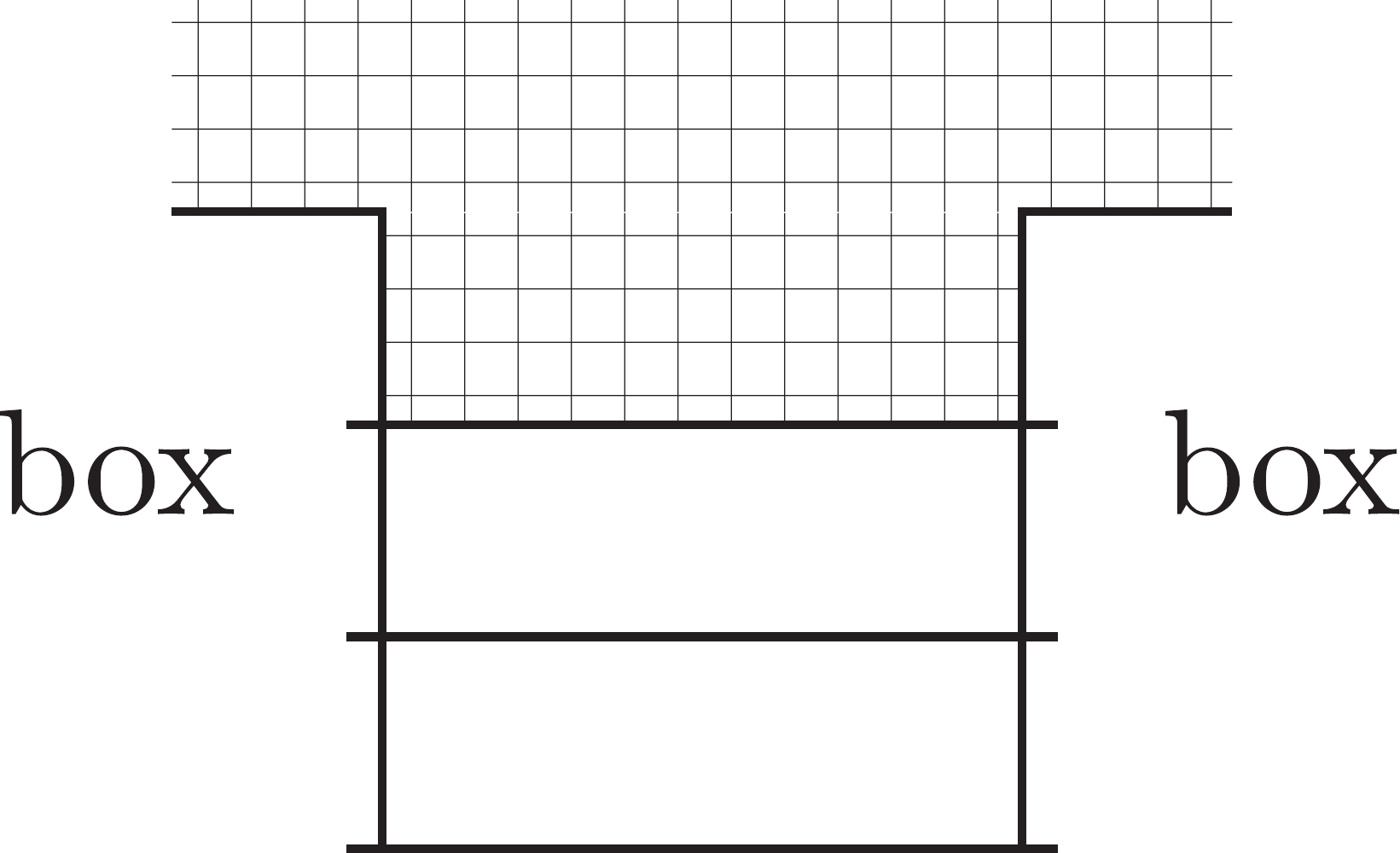}
\caption{A region having one side in $S^2 \setminus (B_1 \cup B_2 \cup \dots \cup B_{2n})$}\label{1308a}
\end{figure}
\item    
\begin{figure}
\includegraphics[width=5cm]{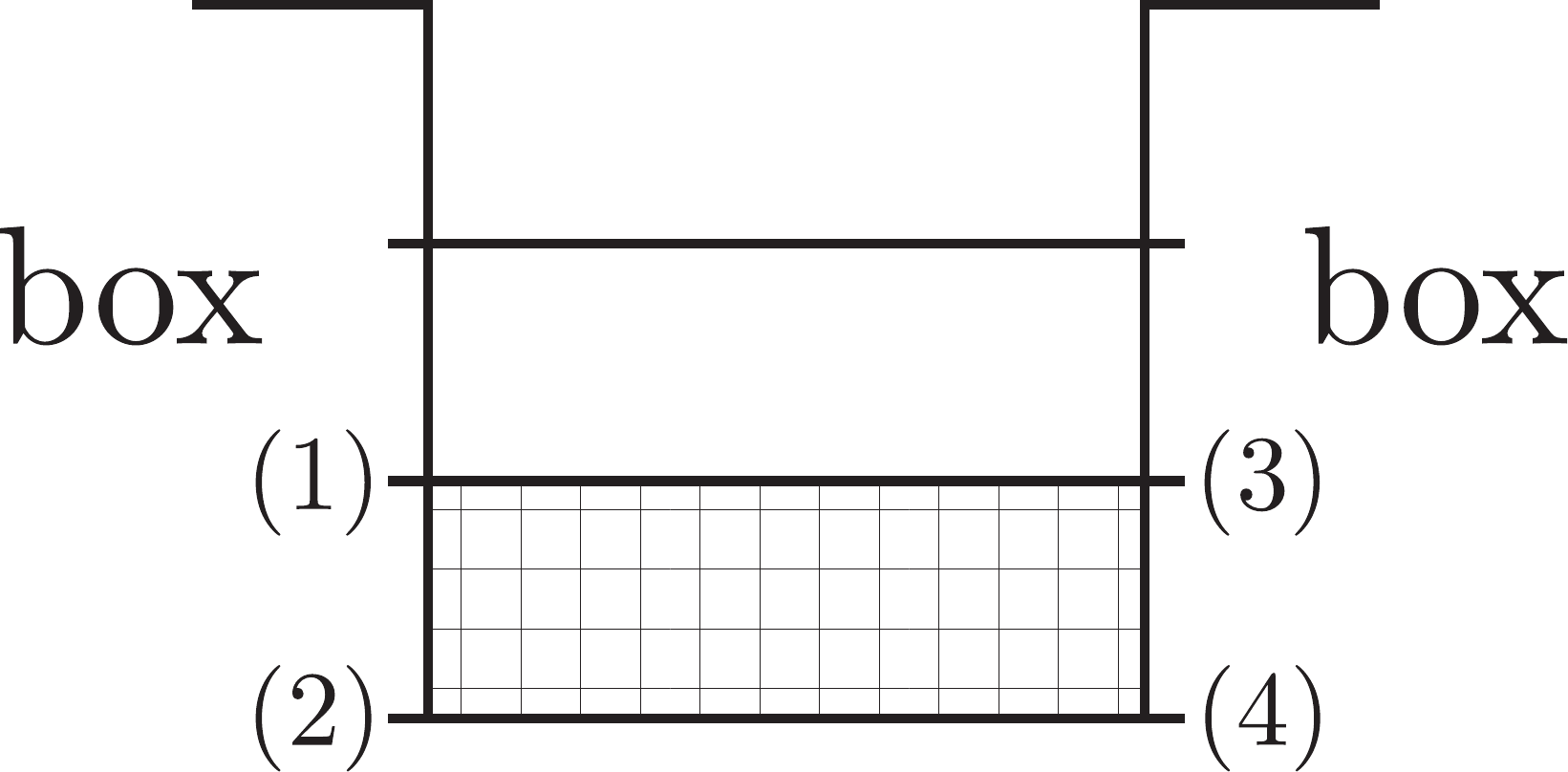}
\caption{Each of endpoints (1), (2), (3), and (4) connects to strand~1, 2, or 3.}\label{1308b}
\end{figure}
Suppose that there exists the $1$-gon  containing a region having two sides outside the boxes, as shown in Fig.~\ref{1308b}.  If one of the two sides is directly connected to either strand~1 or 2, then the $1$-gon has at least two double points, which is a contradiction (there is no $1$-gon with two double points).  If one of the two sides is directly connected to strand~3 that is connected strand~1 or 2 in the adjacent box, then this also implies a contradiction which is similar to the case above.     
\end{itemize}  
\subsection{On a deformation of type~3 not occurring within a box}
If $Op_{r}$ is a single deformation of type~3, we focus on the $3$-gon $T_{r-1}$ in $Q_{r-1}$ with respect to the single deformation of type~3, as shown in Fig.~\ref{1309}. 

\begin{itemize}
\item Suppose that $T_{r-1}$ contains a region having one side, as shown in Fig.~\ref{1308a}.  By the induction assumption of $Q_{r-1}$, the region has at least four sides,  which is a contradiction.    
\begin{figure}[h!]
\includegraphics[width=5cm]{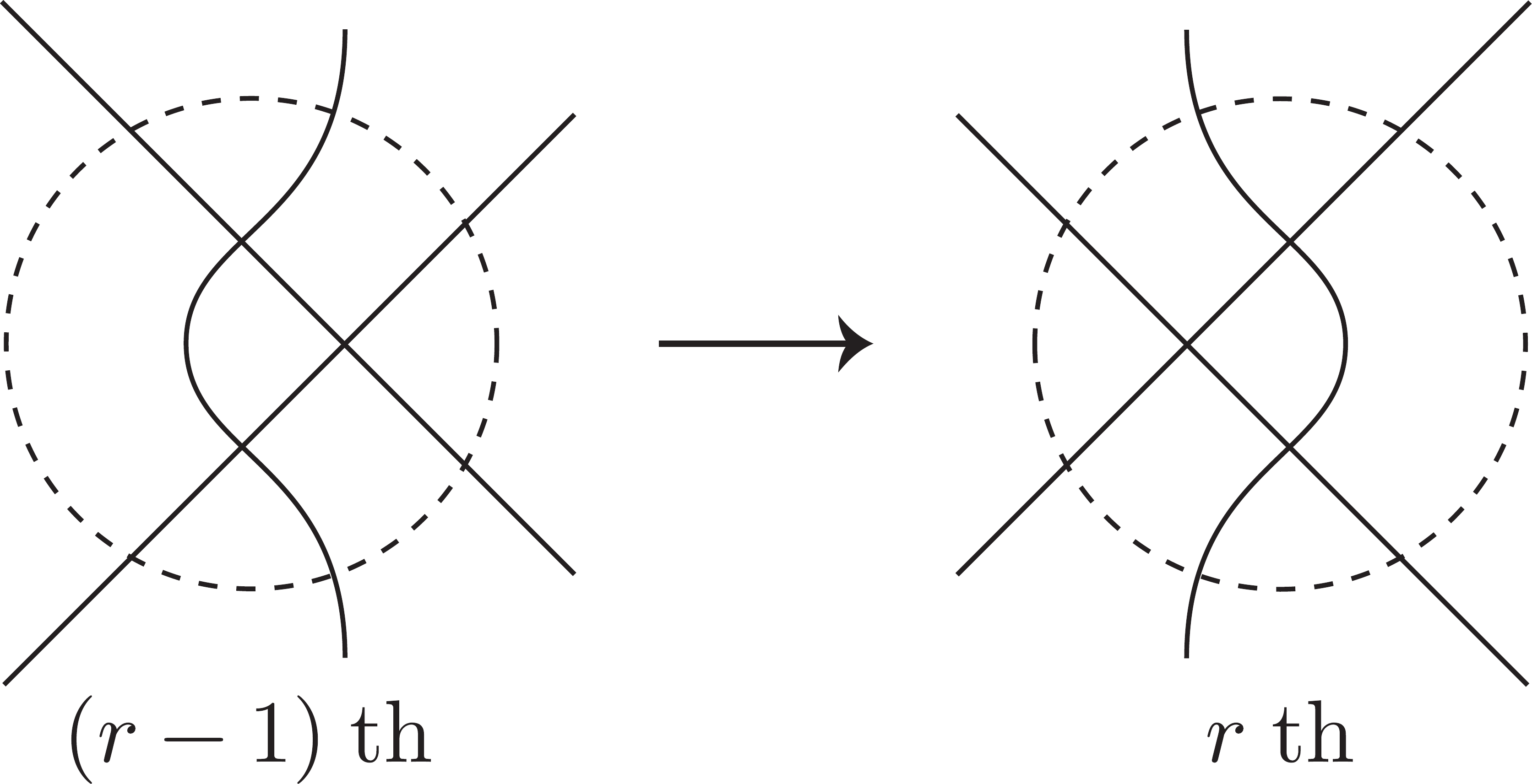}
\caption{$3$-gon $T_{r-1}$ in $Q_{r-1}$ with respect to a single deformation of type~3}\label{1309}
\end{figure}
\item Suppose that $T_{r-1}$ contains a region having two sides outside the boxes, as shown in Fig.~\ref{1310}.  
\begin{figure}[h!]
\includegraphics[width=3cm]{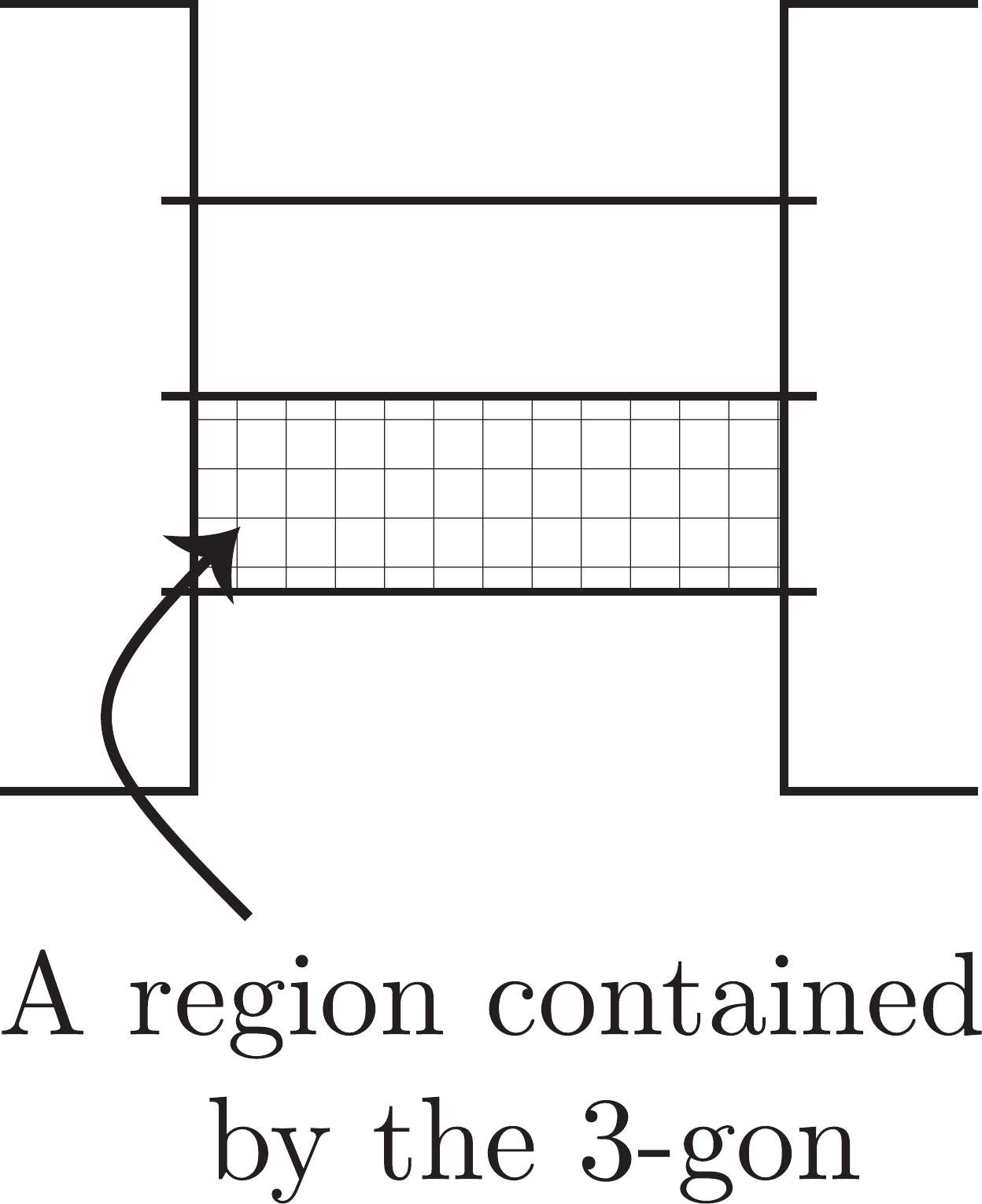}
\caption{A region surrounded by the $3$-gon $T_{r-1}$ appearing in $Q_{r-1}$}\label{1310}
\end{figure} 
\begin{enumerate}
\item Suppose that no double point of $T_{r-1}$ is  in a box, as shown in Fig.~\ref{1311}.  However, this situation is prohibited by (B) of ($2m, 2n$) box property.  
\begin{figure}[h!]
\includegraphics[width=5cm]{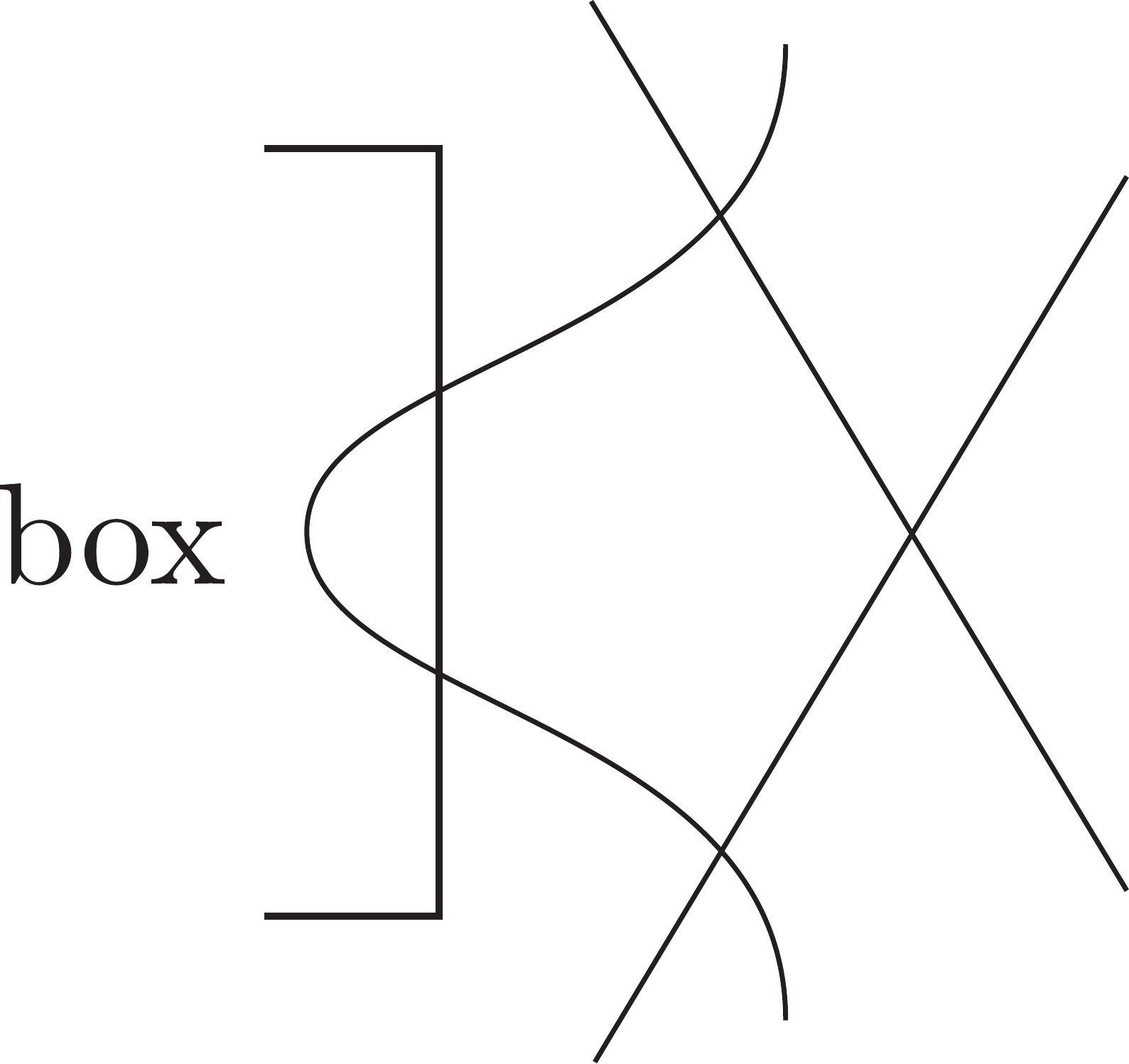}
\caption{Impossible case}\label{1311}
\end{figure}

\item Suppose that $T_{r-1}$ has at least one double point in a box.    
By the induction assumption, there is no double point outside the boxes for $Q_{r-1}$.  Thus, if $T_{r-1}$ is not inside a box, there are exactly two cases.  
\begin{enumerate}
\item Case~1. A double point of $T_{r-1}$ is inside a box and the other two double points of $T_{r-1}$ are in another box, as shown in the left figure of Fig.~\ref{1312}.  
\item Case~2. The three double points of $T_{r-1}$ are in three different boxes, respectively, as shown in the right figure of Fig.~\ref{1312}.  
\end{enumerate}  
\begin{figure}[h!]
\includegraphics[width=7cm]{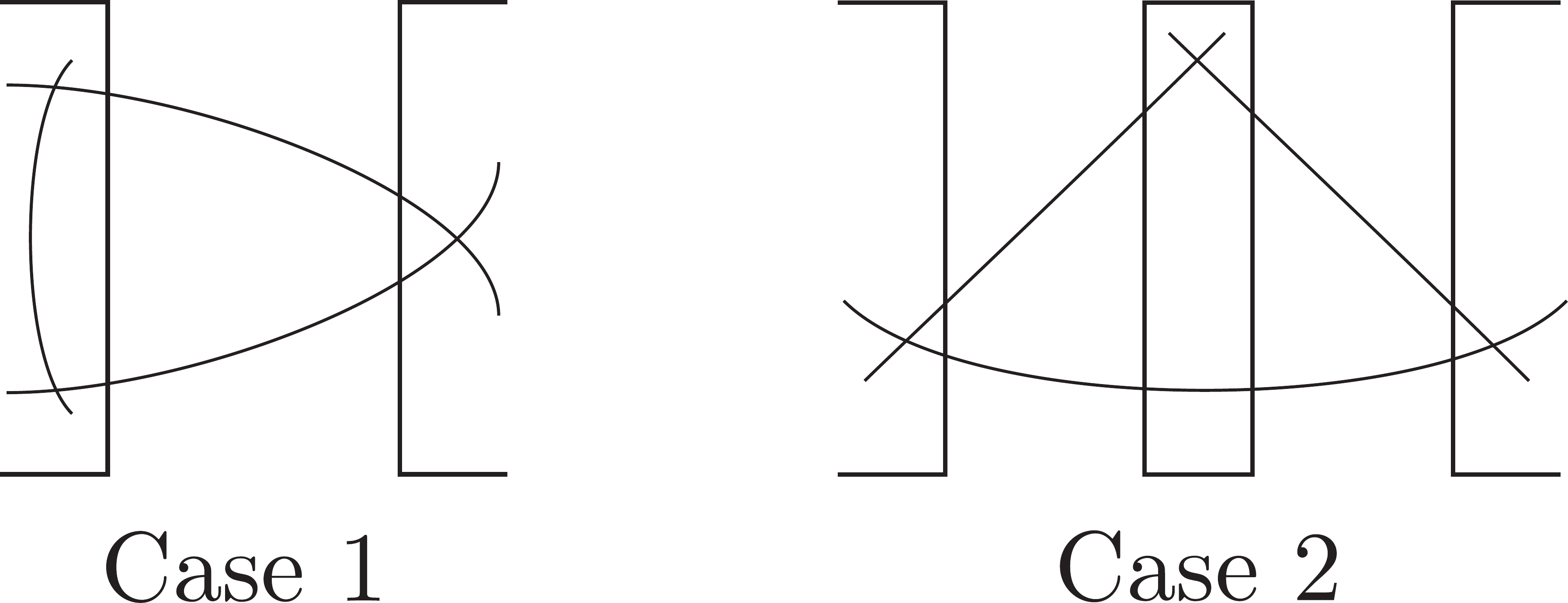}
\caption{Case~1 (left) and Case~2 (right)}\label{1312}
\end{figure}
As a result, for Case~1 (Case~2,~resp.), by retaking one box (two boxes,~resp.), as shown in Fig.~\ref{1313}, $T_{r-1}$ is contained entirely within a box.  Note that if $Q_{r-1}$ satisfies (B) of ($2m, 2n$) box property, then $Q_{r}$ satisfies (B) of ($2m, 2n$) box property since every connection preserves in each box and the positions of endpoints on $\partial B_i$ also preserve up to ambient isotopy  under the deformation.  
\begin{figure}[h!]
\includegraphics[width=10cm]{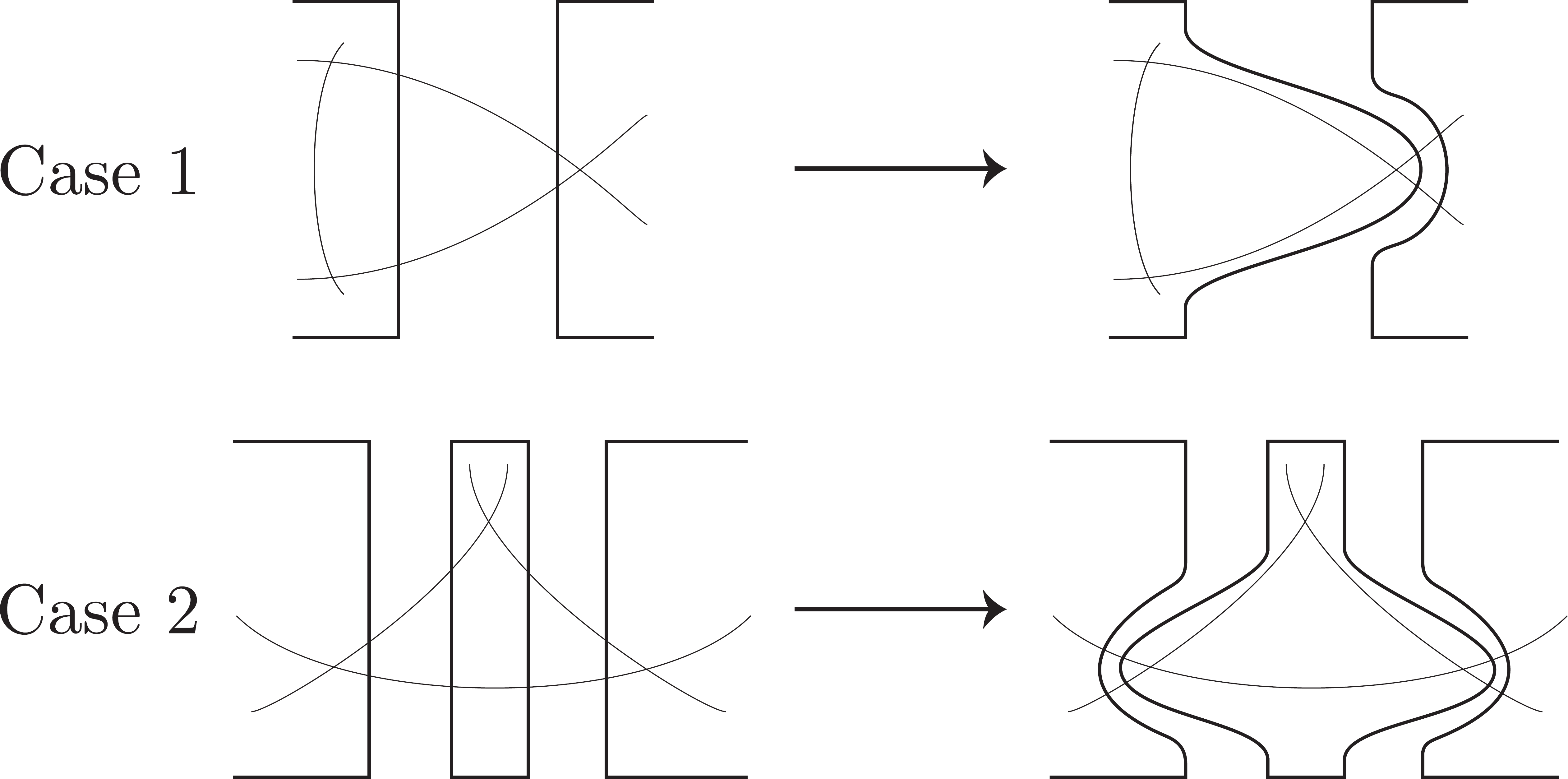}
\caption{Retaking box(es) for Case~1 (upper) and Case~2 (lower).  The positions of endpoints on $\partial B_i$ preserve up to ambient isotopy under the deformation.}\label{1313}
\end{figure}
\end{enumerate}
\end{itemize}
This completes the proof of Claim~\ref{claim_A}, which implies that $\rii(P(m, n)) \ge 1$.  

\subsection{Proof of $\rii(P(m, n)) \ge m$.}\label{36}
Recall Claim~\ref{claim_A}.    
By the argument of the above proof of Claim~\ref{claim_A},  we have Lemma~\ref{lem5}.  
\begin{lemma}\label{lem5}
Let $k$ be an integer $(0 \le k \le r)$.  
Let $Q_0$ $=$ $P(m, n)$ and $Q_k$ be the knot projection obtained from $P(m, n)$ by $Op_k$ using Notation~\ref{notation1}:  
\[
P(m, n) = Q_0 \stackrel{Op_1}{\to} Q_1 \stackrel{Op_2}{\to} \dots \stackrel{Op_r}{\to} Q_r,   
\]
where the sequence consists of a single deformation of negative type 2 and deformations of type 1 and 3.  
If $Op_{k+1}$ is a deformation of negative type 2 within a box, then we have the following statements.  
\begin{enumerate}
\item Each of $\{Q_0, Q_1, \ldots, Q_{k} \}$  preserves $(2m, 2n)$ box property.  \label{statement1}
\item  $Q_{k+1}$ satisfies $(2m-2, 2n)$ box property.  \label{statement2}
\item Each of $\{Q_{k+2}, Q_{k+3}, \ldots, Q_r \}$ preserves $(2m-2, 2n)$ box property.    \label{statement3}
\end{enumerate}
If a deformation $Op_{k+1}$ of negative type 2 is not within a box, then, by retaking boxes, the case returns to the case that $Op_{k+1}$ is a deformation of negative type 2 within a box.  
\end{lemma}
\begin{proof}
Suppose that $Op_{k+1}$ is a deformation of negative type 2 within a box.  

\noindent (\ref{statement1}) ((\ref{statement3}),~resp.) By using the argument of the proof of Claim~\ref{claim_A}, we may suppose that every $Op_j$ ($j \neq k+1$) is a deformation of type 1 or 3 within a box by retaking boxes.  Since $Op_j$ of type 1 or 3 within a box preserve ($2m, 2n$) (($2m-2, 2n$),~resp.) box property, the statement (\ref{statement1}) ((\ref{statement3}),~resp.) holds.   

\noindent (\ref{statement2}) The deformation of negative type 2 decreases or preserve double points formed by a subarc of 1 and a subarc of 2.  It implies the statement (\ref{statement2}).  

Suppose that a deformation $Op_{k+1}$ of negative type 2 is not within a box.  
Then, by retaking one box (two boxes,~resp.), as shown in Fig.~\ref{negative2}, the case returns to the case that $Op_{k+1}$ is a deformation of negative type 2 within a box.  
\end{proof}
\begin{figure}[h!]
\includegraphics[width=8cm]{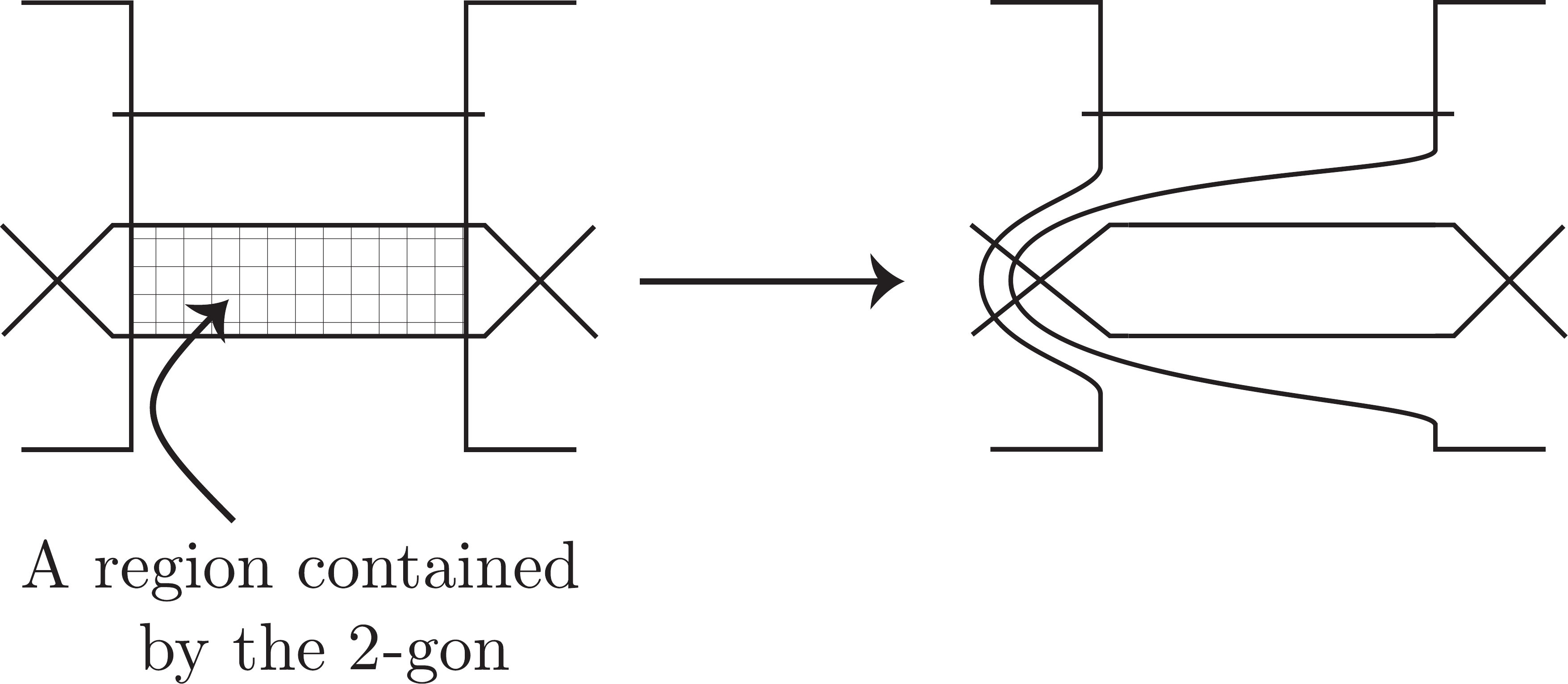}
\caption{A bigon to which will be applied a deformation of negative type~2 not occurring within a box}\label{negative2}
\end{figure}
Lemma~\ref{lem5} immediately implies $\rii (P(m, n)) \ge m$.     

\subsection{Proof of $P(m, n) \le m$.}
We prepare Lemma~\ref{lem4}.  
\begin{lemma}\label{lem4}
Let $k$ be a positive integer.  Each of replacements $T(2k-1)$ and $T(2k)$ as in Fig.~\ref{tangle1i} is always possible under (1, 3) homotopy.   
\newline
\begin{figure}[h!]
\begin{center}
\includegraphics[width=10cm]{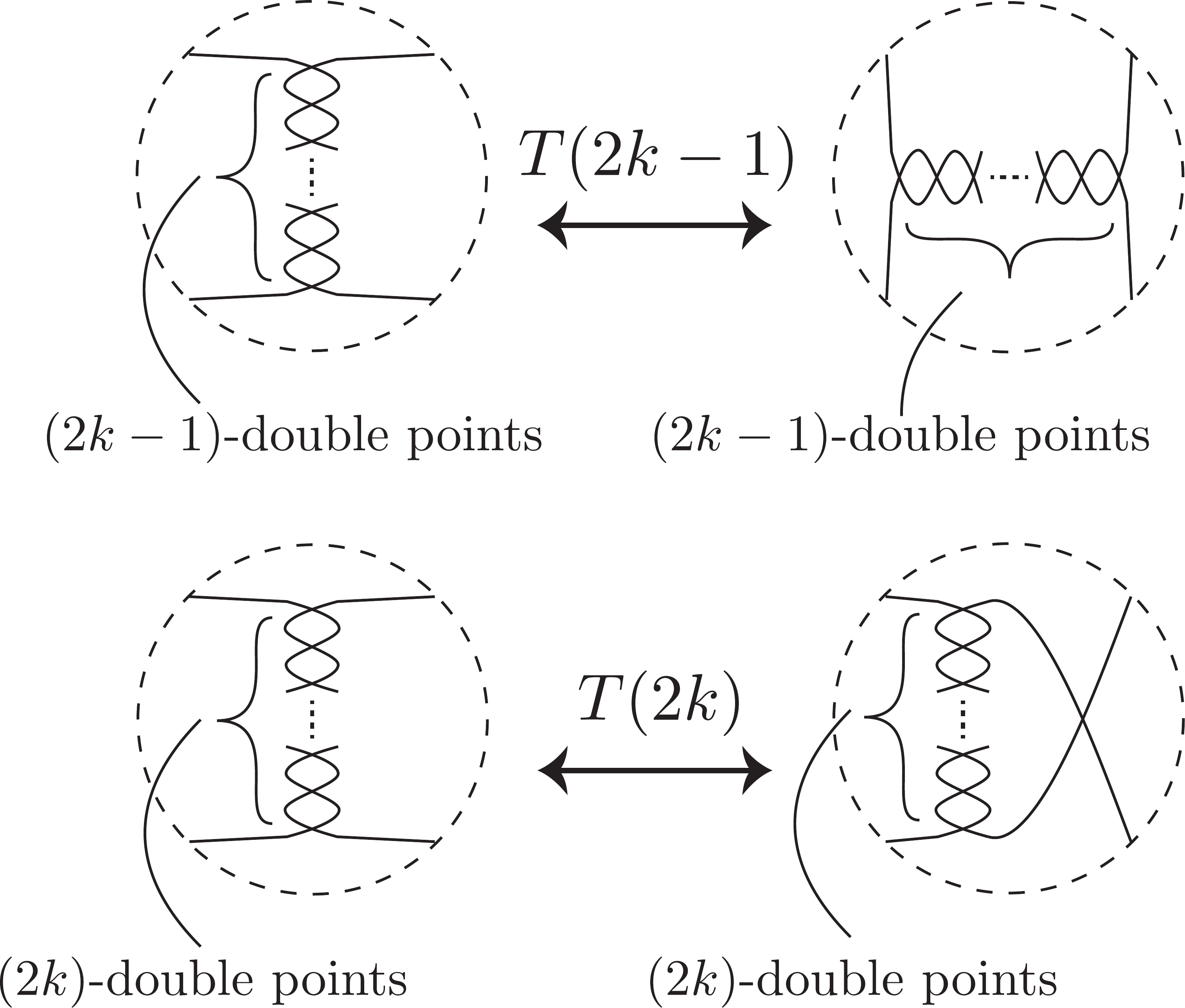}
\end{center}
\caption{Definitions of $T(2k-1)$ and $T(2k)$}\label{tangle1i}
\end{figure}
\end{lemma}

\begin{proof}
We prove the statement of the lemma by the induction.  
\begin{itemize}
\item Case $k=1$: for $T(1)$, the statement is clear.  $T(2)$ holds as in Fig.~\ref{lemma1b2i}.  
\begin{figure}[h!]
\begin{center}
\includegraphics[width=8cm]{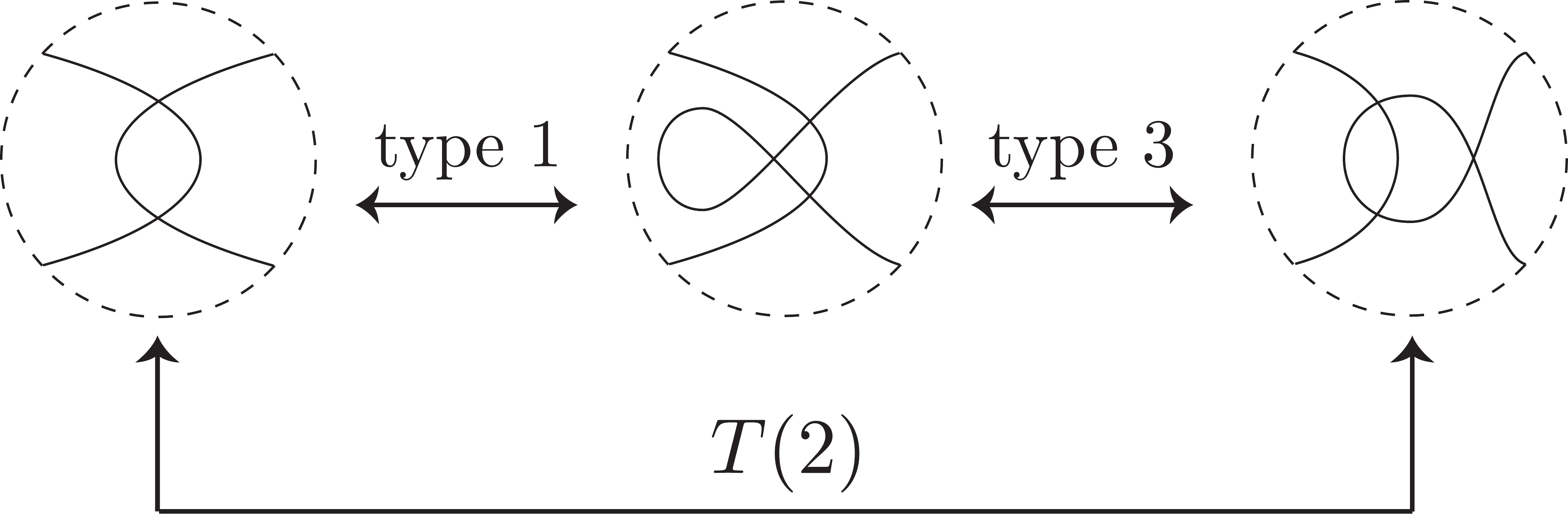}
\end{center}
\caption{$T(2)$}\label{lemma1b2i}
\end{figure}
\item Case $k=i$: Suppose that $T(2i-2)$ is a possible local replacement under (1, 3) homotopy.  In this case, $T(2i-1)$ holds, which also implies $T(2i)$.  See Fig.~\ref{tangle2i}.    
\begin{figure}[h!]
\begin{center}
\includegraphics[width=12cm]{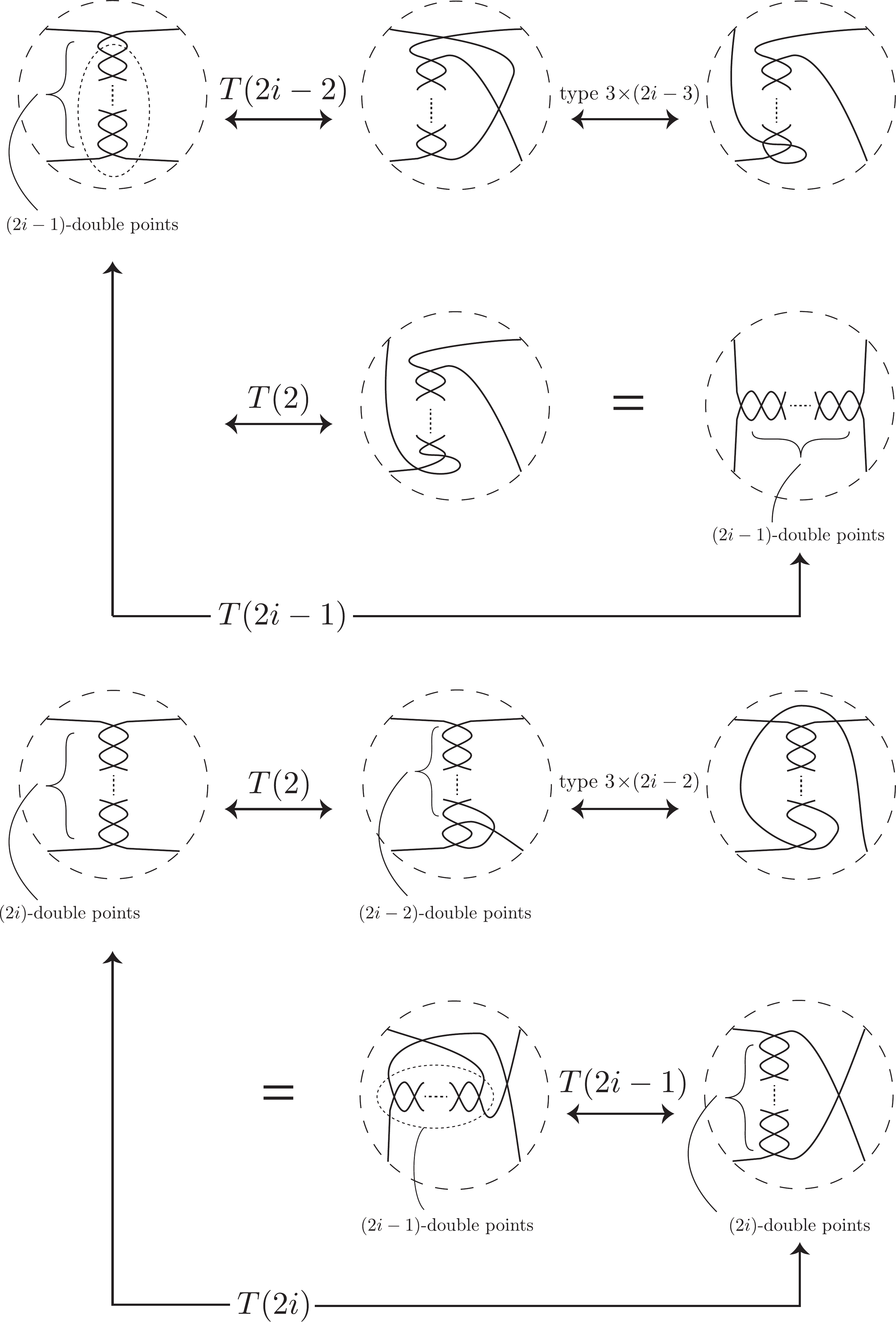}
\end{center}
\caption{An inductive argument of $T(2i-1)$ and $T(2i)$}\label{tangle2i}
\end{figure}
\end{itemize}
It completes the induction step.  
\end{proof}

\noindent $\bullet$ {\it{Proof of $\rii(P(m, n)) \le m$.}}
Let $x_1$ ($x_k$,~resp.) be the twisting part, including exactly $2m$ double points, corresponding to the first ($k$-th,~resp.) box from the bottom of Fig.~\ref{tangle3i}.   First, we apply $m$ deformations of negative type~2 to the twisting part in a single box $x_1$, which erases $x_1$.   
Second, by applying $T(2m)$ to $x_3$ and the leftmost single double point of $x_2$, the number of double points in $x_2$ decreases by $1$.  Repeating this argument $2m-1$ times, we complete erasing a single box $x_2$.   
By applying $T(2m)$ to $x_4$ and the rightmost  single double point of $x_3$, the number of double points in $x_3$ decreases by $1$.  Repeating this argument $2m-1$ times, we complete erasing a single box $x_3$.   
We apply the argument to each $x_k$, which erase $2n-2$ boxes.   
\begin{figure}[h!]
\begin{center}
\includegraphics[width=10cm]{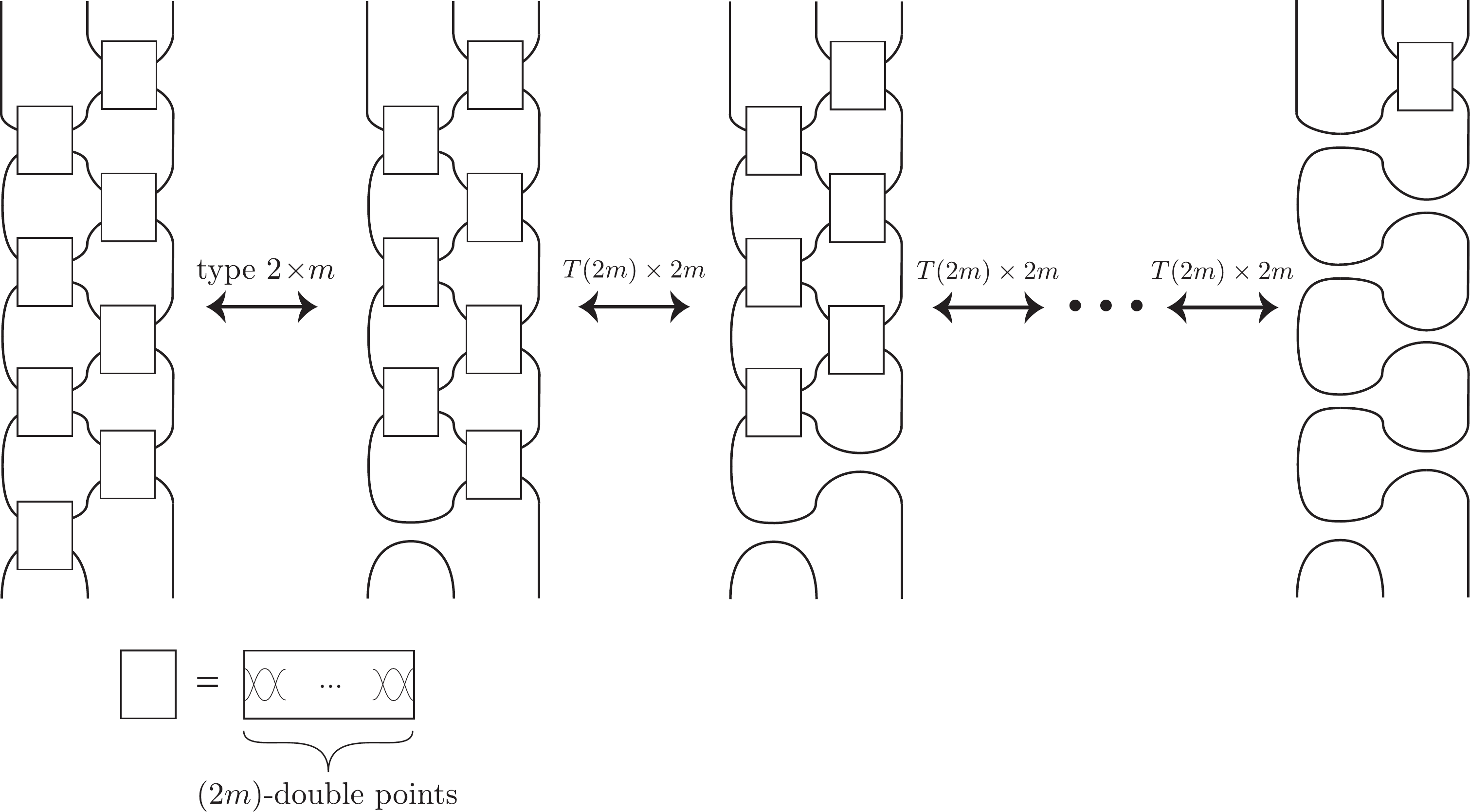}
\end{center}
\caption{Applications of Lemma~\ref{lem4}}\label{tangle3i}
\end{figure}
Finally, we apply $2 m$ deformations of type~1 as in Fig.~\ref{tangle4i}.  

\begin{figure}[h!]
\begin{center}
\includegraphics[width=6cm]{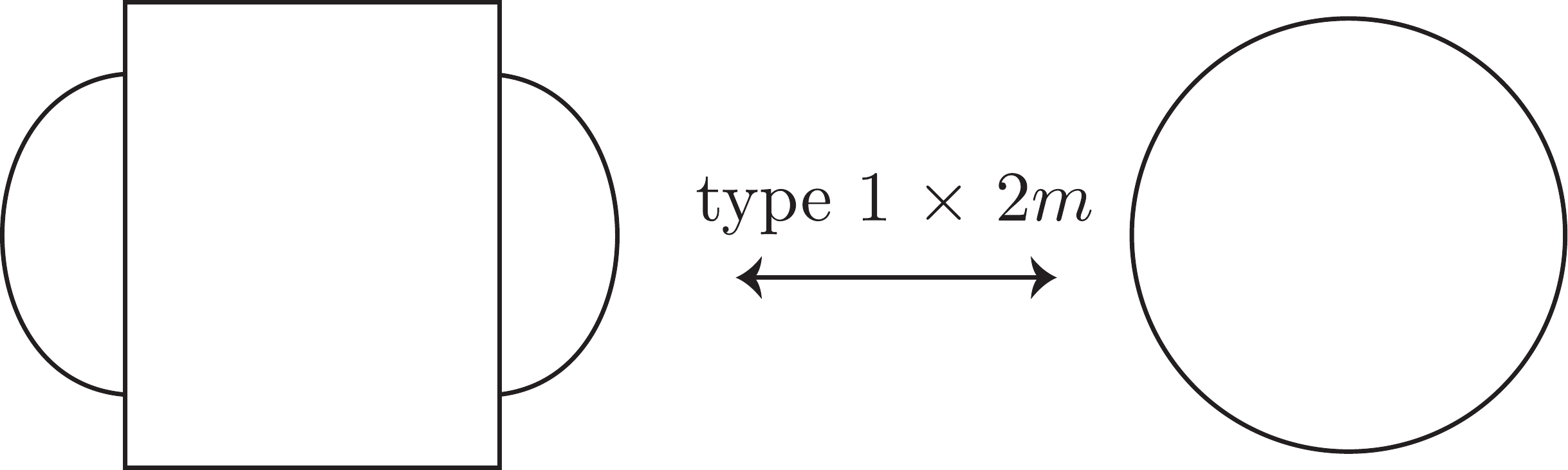}
\end{center}
\caption{$2m$ deformations of type 1}\label{tangle4i}
\end{figure}

$\hfill\Box$

\begin{remark}
The operation $T(n)$ ($n \in \mathbb{Z}_{>0}$) of Lemma~\ref{lem4} corresponds to a generalization of a type of an  edge of a complex \cite{FHIKM} or the operation $\alpha$ \cite{HI}.      
\end{remark}


\section{Applications}\label{application}
\subsection{Every pretzel knot projection is (1, 3) homotopic to the trivial spherical curve.}\label{secA1}
\begin{notation}
A part consisting of $m$ double points of a knot projection as in Fig.~\ref{1303}, it is called a \emph{twist}.  
\end{notation}
\begin{proposition}\label{proposition1}
Every pretzel knot projection as in Fig.~\ref{pre} is (1, 3) homotopic to the trivial spherical curve.  
\end{proposition}
\begin{proof}
Let $P(a_1, a_2, \ldots, a_n)$ be a knot projection as in Fig.~\ref{pre}, where each $a_i$ presents double points of a twist ($1 \le i \le n$).  
\begin{figure}[h!]
\includegraphics[width=5cm]{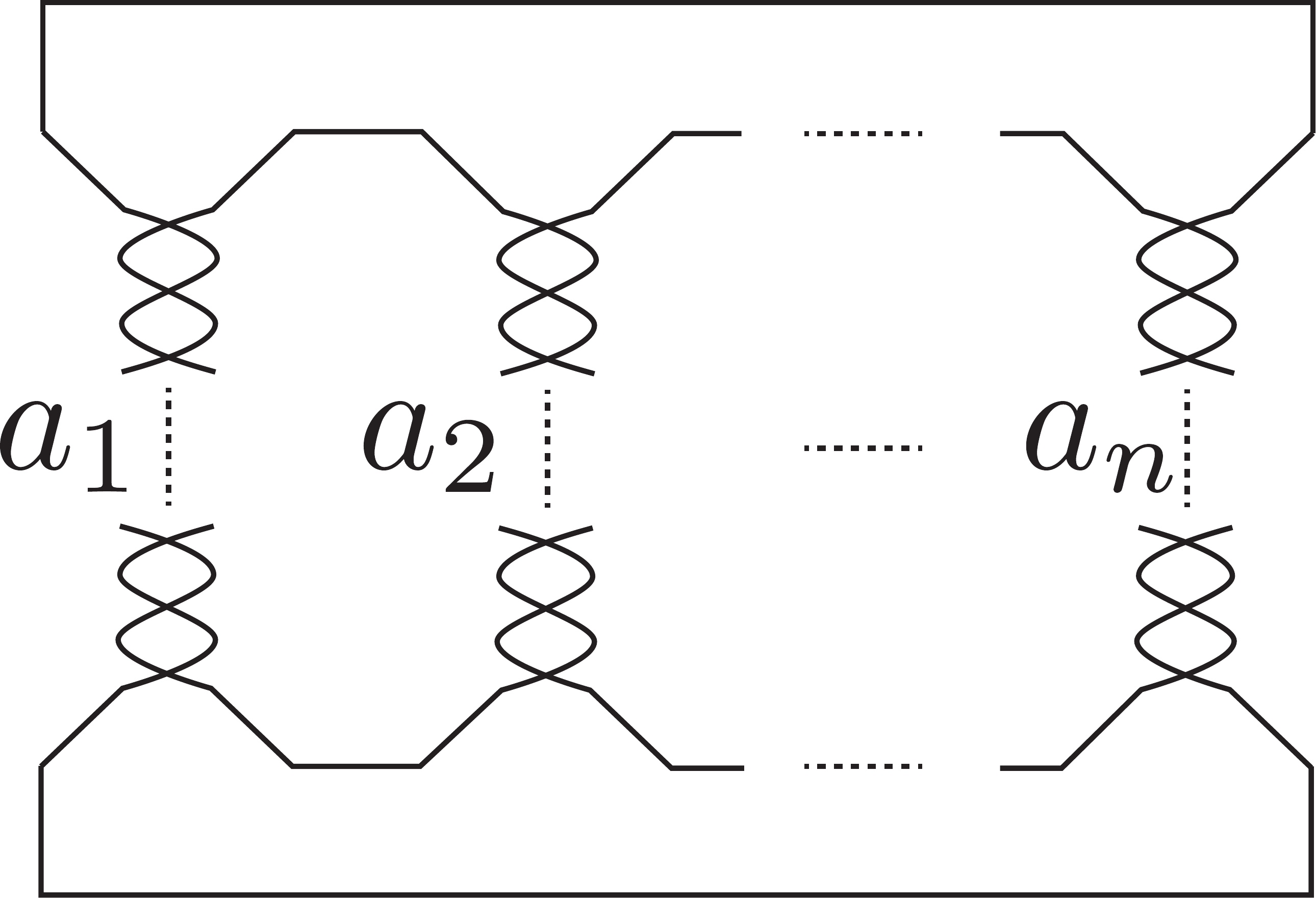}
\caption{Pretzel knot projection $P(a_1, a_2, \dots, a_n)$}\label{pre}
\end{figure}

It is sufficient to consider the two cases.  
\begin{itemize}
\item Case~1: for each $i$, $a_i$ are odd double points and $n$ is an odd positive number.    
\item Case~2: $a_1$ are even double points and for $i \neq 1$, $a_i$ are odd double points.     
\end{itemize}

\begin{figure}[h!]
\includegraphics[width=2cm]{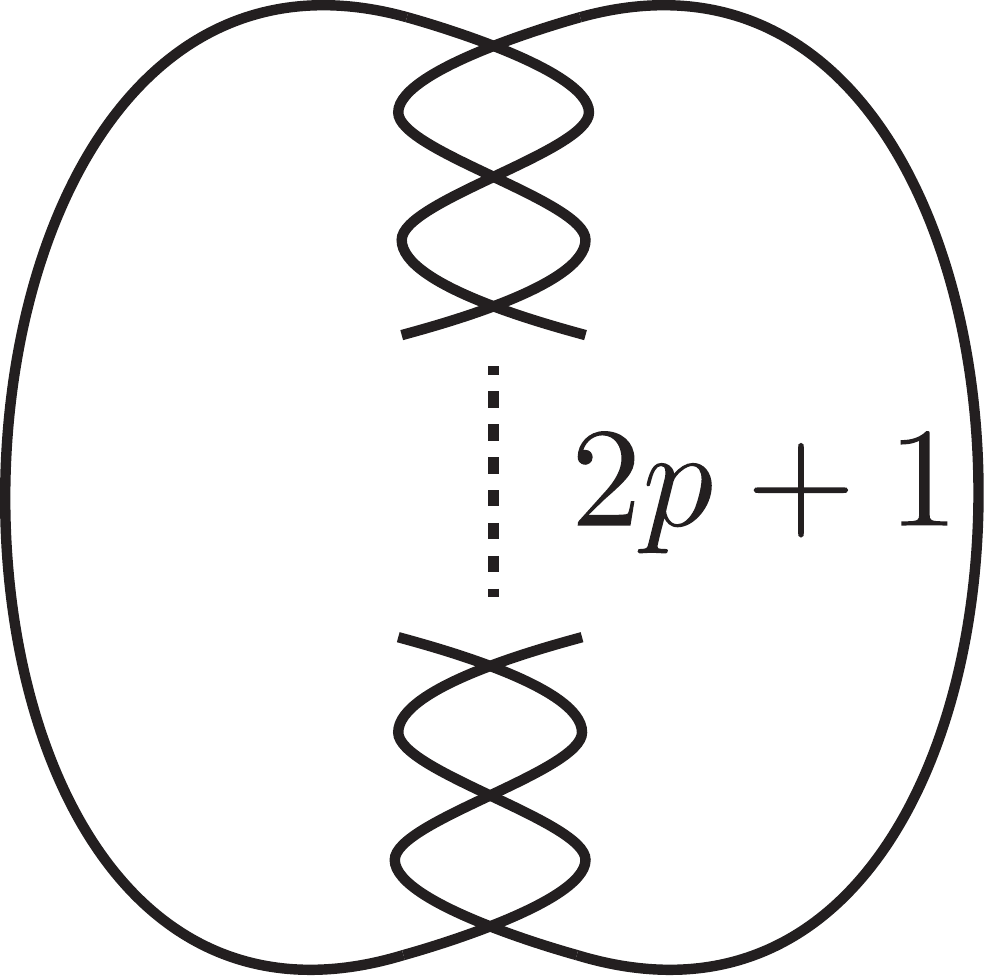}
\caption{$(2, 2p+1)$-torus knot projection}\label{torus}
\end{figure}

It is easy to see that every $(2, 2p+1)$-torus knot projection ($p \in \mathbb{N}$), as shown in Fig.~\ref{torus}, is (1, 3) homotopic to the trivial spherical curve by applying $T(2p+1)$ and applying deformations of type 1 decreasing double points.  

For Case~1, we apply $T(a_i)$ to each twist, and we have a $(2, 2p+1)$-torus knot projection where $2p+1$ $=$ $\sum_{i=1}^{n} a_i$.  For Case~2, since $a_i$ ($i \neq 1$) is an odd number, we apply $T(a_i)$ to each twist corresponding to $a_i$ double points ($i \neq 1$), and apply $T(a_1)$ ($\sum_{i=2}^{n} a_i$ times) to the resulting knot projection, which implies a knot projection having exactly $a_1$ double points.  After that, deformations of type 1 decreasing double points obtain the trivial spherical curve.  
\end{proof}
\subsection{Every two-bridge knot projection is (1, 3) homotopic to the trivial spherical curve.}\label{secA2}
\begin{proposition}\label{proposition2}
Every two-bridge knot projection as in Fig.~\ref{ratio} is (1, 3) homotopic to the trivial spherical curve.  
\end{proposition}
\begin{figure}[htbp]
\includegraphics[width=8cm]{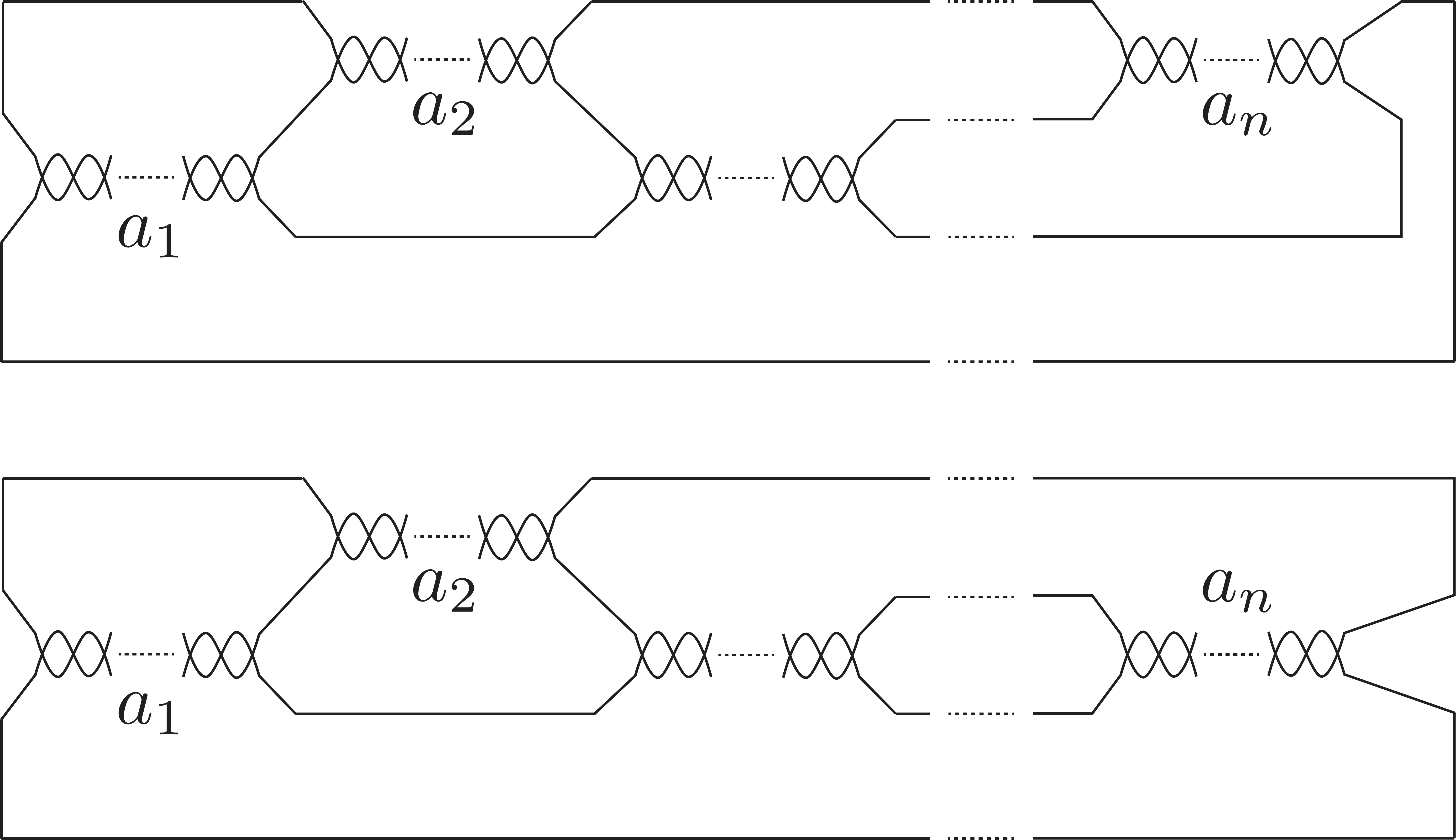}
\caption{Two-bridge knot projection $P(a_1, a_2, \ldots, a_n)$}\label{ratio}
\end{figure}

\begin{proof}
Let $P(a_1, a_2, \ldots, a_n)$ be a knot projection as in Fig.~\ref{ratio}, where each $a_i$ presents double points of a twist ($1 \le i \le n$).  

\begin{itemize}
\item Case~1: Suppose that $a_1$ is an odd number corresponding to $a_1$ double points.  By applying a single $T(a_1)$, two twists $a_1$ and $a_2$ merge a twist consisting of $a_1 + a_2$ double points.   
\item Case~2: Suppose that $a_1$ is an even number corresponding to $a_1$ double points.  By applying $T(a_1)$ $a_2$ times, $a_2$ double points are resolved, which implies a twist disappears.  
\end{itemize}
By an induction of the number of twists, we resolve every twist.  
\end{proof}
\begin{remark}
By using Lemma~\ref{lem4}, for a knot projection $P$ consisting of two-bridge knot projections via tangle sums or  connected sums, it is elementary to show that such $P$ is (1, 3) homotopic to the trivial spherical curve by using Propositions~\ref{proposition1} and \ref{proposition2}.     
\end{remark}

\section*{Acknowledgements}
We would like to explain the detail of the paper.     
First, Professor Tsuyoshi Kobayashi shared with a progress of the study of Ms.~Sumika Kobayashi for her Master thesis, on September 20, 2018.  This work included  $T(2k-1)$.    
Then, we sent a preprint (corresponding to an earlier version of this paper) with a slide of a talk in a Colloquium of Department of Mathematics, Faculty of Education, Waseda University on December 18, 2015, which includes Lemma~\ref{lem4}.  
The motivations of the two works, \cite{Ks, KK} and this paper are different.  
The author would like to thank Professor Tsuyoshi Kobayashi for encouraging us to write this paper and giving us an information of \cite{Ks, KK}.   Ms.~Sumika Kobayashi gave another application of Lemma~\ref{lem4} by other motivation of hers \cite{Ks, KK}.

The authors also would like to thank the referee for the comments.  

\end{document}